\documentclass[11pt,reqno]{amsart}
\usepackage{amssymb,mathrsfs,bm,enumerate}
\usepackage{amsfonts,dsfont,%txfonts
mathptmx,amsaddr}

\makeatletter
\newtheorem{theorem}{Theorem}[section]
\newtheorem{lemma}[theorem]{Lemma}
\newtheorem{corollary}[theorem]{Corollary}
\newtheorem{proposition}[theorem]{Proposition}
\theoremstyle{definition}
\newtheorem{remark}[theorem]{Remark}
\newtheorem{definition}[theorem]{Definition}

\numberwithin{equation}{section}

\def\Ric{{\operatorname{Ric}}}

 \def\Cut{{\operatorname{Cut}}}

\newcommand\1{\hbox{\kern.375em\vrule height1.57ex depth-.1ex
    width.05em\kern-.375em \rm 1}}

 \newcommand\E{\mathbb{E}}
 
\newcommand\N{\mathbb{N}} \newcommand\R{\mathbb{R}} \newcommand\Z{\mathbb{Z}}
\renewcommand\P{\mathbb{P}}

\newcommand\newdot{{\kern.8pt\cdot\kern.8pt}}
\def\mathpal#1{\mathop{\mathchoice{\text{\rm #1}}%
    {\text{\rm #1}}{\text{\rm #1}}%
    {\text{\rm #1}}}\nolimits} \def\id{{\mathpal{id}}}
 \def\OtM{{\mathpal{O}}_t(M)}
\def\FM{\mathpal{F}(M)} \def\FM{F(M)}

 \def\vd{\mathrm{d}} \def\r{\right}
\def\l{\left} \def\e{\operatorname{e}} \def\dsum{\displaystyle\sum}

  \makeatother

\begin{document}
%%%%%%%%%%%%%%%%%%%%%%%%%%%%%%%%%%%%%%%%%%%%%%%%%%%%%%%%%%%%%%%%%%%%%%%

  \title[Evolution systems of measures and semigroup properties]
  {Evolution systems of measures and semigroup properties\\ on
    evolving manifolds}

  \author[L.-J.~Cheng and A.~Thalmaier]{Li-Juan
    Cheng\textsuperscript{1,2} and Anton Thalmaier\textsuperscript{1}}

  \address{\textsuperscript{1}Mathematics Research Unit, FSTC, University of Luxembourg,\\
    Maison du Nombre, 4364 Esch-sur-Alzette, Grand Duchy of
    Luxembourg}
  \address{\textsuperscript{2}Department of Applied Mathematics, Zhejiang University of Technology,\\
    Hangzhou 310023, The People's Republic of China}
  \email{lijuan.cheng@uni.lu \text{and} chenglj@zjut.edu.cn}
  \email{anton.thalmaier@uni.lu}

  \begin{abstract}
    An evolving Riemannian manifold $(M,g_t)_{t\in I}$ consists of a
    smooth $d$-dimensional manifold $M$, equipped with a geometric
    flow $g_t$ of complete Riemannian metrics, parametrized by
    $I=(-\infty,T)$. Given an additional $C^{1,1}$ family of vector
    fields $(Z_t)_{t\in I}$ on $M$. We study the family of operators
    $L_t=\Delta_t +Z_t $ where $\Delta_t$ denotes the Laplacian with
    respect to the metric $g_t$.  We first give sufficient conditions,
    in terms of space-time Lyapunov functions, for non-explosion of
    the diffusion generated by $L_t$, and for existence of evolution
    systems of probability measures associated to it.  Coupling
    methods are used to establish uniqueness of the evolution systems
    under suitable curvature conditions.  Adopting such a unique
    system of probability measures as reference measures, we
    characterize supercontractivity, hypercontractivity and
    ultraboundedness of the corresponding time-inhomogeneous
    semigroup. To this end, gradient estimates and a family of
    (super-)logarithmic Sobolev inequalities are
    established.  \end{abstract}

  \keywords{Evolution system of measures, geometric flow,
    inhomogeneous diffusion, semigroup, supercontractivity,
    hypercontractivity, ultraboundedness} \subjclass[2010]{60J60,
    58J65, 53C44} \date{\today}

  \maketitle
  % \tableofcontents

\section{Introduction}\label{Sect1}
Let $M$ be a $d$-dimensional differentiable manifold equipped with a
family of complete Riemannian metrics $(g_t)_{t\in I}$ which is $C^1$
in $t$ and evolves according to
$$\frac{\partial}{\partial t}g_t=2h_t,\quad t\in I,$$
where $I=(-\infty,T)$ for some $T\in (-\infty,+\infty]$, and $h_t$ a
time-dependent 2-tensor on $TM$.  Denote by $\nabla^t$, $\Delta_t$ the
Levi-Civita connection, resp. Laplacian on $M$, both with respect to
the metric~$g_t$.  For a given $C^{1,1}$ family $(Z_t)_{t\in I}$ of
vector fields on $M$, we study the time-dependent second order
differential elliptic operator $L_t=\Delta_t+Z_t$.

In this paper, we develop the basis for a general theory of the
following backward Cauchy problem:
\begin{equation}\label{heat-1}
  \left\{  \begin{aligned}
      &\partial_su(\newdot, x)(s)=-L_su(s,\newdot)(x) \\
      &u(t,x)=\phi(x)
    \end{aligned}\right\}, \quad (s,t)\in \Lambda,\ x \in M,
\end{equation}
where $\phi\in C^2(M)\cap C_b(M)$ and
$\Lambda:=\{(s,t): s\leq t\ \text{and}\ s,t\in I\}$.

We investigate this problem from a probabilistic point of view.  Let
$X_t$ be the diffusion process generated by $L_t$ (called
$L_t$-diffusion) which is assumed to be non-explosive before time $T$
(see \cite{ACT,C11,KP11} for details).  As in the time-homogeneous
case, we construct $L_t$-diffusions $X_t$ via horizontal diffusions
$u_t$ above $X_t$.

Let $\FM$ be the frame bundle over $M$ and $\OtM$ the orthonormal
frame bundle with respect to the metric $g_t$.  We denote by
$\pi\colon \FM\rightarrow M$ the projection from $\FM$ onto $M$.  For
a frame $u\in \OtM$, denote by $H^{t}_{Y}(u)$ the
$\nabla^{t} $-horizontal lift of $Y\in T_{{\pi}u}M$.  This allows one
to determine standard-horizontal vector fields $H_{i}^t$ on $\OtM$,
via the formula
$$H_{i}^t(u)=H_{ue_i}^t(u),\quad i=1,2,\ldots,d,$$
where $(e_{i})_{i=1}^{d}$ denotes the canonical orthonormal basis of
$\R^d$.  Furthermore, we denote by
$(V_{\alpha,\, \beta})_{\alpha,\, \beta=1}^d$ the standard-vertical
fields on $\FM$.  Then given $s\geq0$, the diffusion $u_t$ is
constructed for $t\geq s$ as the solution to the following
Stratonovich SDE:
\begin{align}\label{SDE-u}
  \begin{cases}
    \vd u_t=\sqrt2\dsum_{i=1}^{d}H_{i}^t(u_t)\circ \vd B_t^{i}+H_{Z_t
    }^t(u_t)\,\vd t-\frac12\dsum_{\alpha,
      \beta=1}^{d}(\partial_tg_t)(u_te_{\alpha}, u_te_{\beta})V_{\alpha,\, \beta}(u_t)\,\vd t,\medskip\\
    u_s\in {{\mathpal{O}}_s(M)},\ \pi u_s=x,
  \end{cases}
\end{align}
where $B_t$ is a standard Brownian motion on $\R^d$. The
projection $X_t:=\pi u_t$ of $u_t$ onto $M$ then gives the wanted
$L_t$-diffusion process on $M$, see \cite{ACT}. In the next section we
complement existing results on non-explosion of $X_t$ which is a
subject already studied in \cite{KP11}.

The backward Cauchy problem \eqref{heat-1} is the Kolmogorov equation
to the following non-autonomous SDE on $M$:
\begin{align}\label{SDE-X}
  \vd X_t=u_t\circ \vd B_t +Z_t(X_t)\,\vd t,\quad X_s=x.
\end{align}
Denote by $X_{t}^{(s,x)}$ the solution to Eq.~\eqref{SDE-X} which is
assumed to be non-explosive before time $T$. Then function
$$u(s,x):=\E[\phi(X_{t}^{(s,x)})]$$ satisfies Eq.~\eqref{heat-1} and 
gives rise to a family of inhomogeneous Markov evolution operators
$(P_{s,t})_{(s,t)\in \Lambda}$ on $M$:
$$P_{s,t}\phi(x):=\E[\phi(X_t^{(s,x)})]=\E^{(s,x)}[\phi(X_t)].$$

This is completely standard in the case of a fixed metric and a
time-independent operator $L_t=L$ where
$P_{s,t}=P_{t-s}={\e}^{(t-s)L}$ and $L^p$-spaces are taken with
respect to an invariant measure $\mu$, i.e., a Borel probability
measure $\mu$ on $M$ such that
$$\int_M P_t\phi\,\vd \mu=\int_M\phi\,\vd \mu,\quad t>0,\ \phi\in \mathcal{B}_b(M).$$
Under suitable conditions, see \cite{BRW01a,BRW01b}, existence and
uniqueness of the invariant measure can be shown. In this case, $P_t$
extends to a contraction semigroup on $L^p(M,\mu)$ for every
$p\in[1,\infty)$, see e.g. \cite{Bakry97,LO98,RW03,Wang01,Wbook1}.

When it comes to the time-inhomogeneous case, the situation turns out
to be more involved.  For instance, Saloff-Coste and Z{\'u}{\~n}iga
\cite{SZ11b,SZ11a} studied the ergodic behavior of time-inhomogeneous
Markov chains; more sophisticated and strict conditions are required
due to the fact that the generator and the semigroup do not commute
and due to the lack of uniqueness of the invariant measure.  A first
goal will be therefore to construct an evolution system of measures as
a family of reference measures which plays a role similar to the
invariant measure in the time-homogeneous case.

Let us start by reviewing the notion of an evolution system of
measures. A family of Borel probability measures $(\mu_t)_{t\in I}$ on
$M$ is called an evolution system of measures (see~\cite{DaR08}) if
\begin{align}\label{invariant-measure}
  \int_M P_{s,t}\phi\,\vd \mu_s=\int_M \phi\,\vd \mu_t,\quad \phi\in \mathcal{B}_b(M),\ (s,t)\in \Lambda.
\end{align}
Recently, Angiuli, Lorenzi, Lunardi et al.~investigated evolution
system of measures and related topics for non-autonomous parabolic
Kolmogorov equations with unbounded coefficients on $\R^d$ (see
\cite{LLA13,KLL10,LLZ16,LAA10}). For instance, in \cite{KLL10}
sufficient conditions for existence and uniqueness of evolution
systems of measures are given; in \cite{LLA13}, using a unique tight
evolution system of measures as reference measures, hypercontractivity
and the asymptotic behavior are studied; the asymptotics in
time-periodic parabolic problems with unbounded coefficients is
addressed in \cite{LAA10}. All this work motivates us to study
evolution systems of measures on evolving manifolds and to investigate
contractivity properties of the semigroup. Our probabilistic approach
simplifies and extends in particular earlier results obtained by
analytic methods.

We start by formulating some hypotheses which will be needed later
on. Let $\rho_t(x,y)$ be the Riemannian distance from $x$ to $y$ with
respect to the metric $g_t$.  Fixing $o\in M$, we write
$\rho_t(x):=\rho_t(o,x)$ for simplicity. Let $\Cut_t$ be the set of
the cut-locus of $(M,g_t)$. Let
  $$\Cut:=\{(x,t)\colon x\in \Cut_t\}.$$
  At different places in the paper, some of the hypotheses listed
  below will be put in force.  \smallskip

  % \noindent{\bf Hypotheses}\,
  \begin{enumerate}[\bf(H1)]
  \item There exists an increasing function
    $\varphi\in C^2(\R^+)$ such that
$$\lim_{r\rightarrow +\infty}\varphi(r)=+\infty\ \text{ and }\   (L_t+\partial_t)(\varphi\circ \rho_t)(x)\leq m(t),
\quad (x,t)\in M\times I\setminus \Cut,$$
for some continuous function $m$ on $I$.
\item There exists an increasing function $\varphi\in C^2({\R}^+)$ such
  that $\varphi(0)=0$,
$$\lim_{r\rightarrow +\infty}\varphi(r)=+\infty\ \text{ and }\ (L_t+\partial_t)(\varphi\circ \rho_t)(x)\leq a(t)-c(t)(\varphi\circ \rho_t)(x),
\quad (x,t)\in M\times I\setminus \Cut,$$
for some non-negative function $a$ and a function $c$ on $I$ such that
$$H(t):=\int_{-\infty}^t\exp{\l(-\int_r^tc(u)\,\vd u\r)}a(r)\,\vd r<\infty.$$

\item There exists a function $k$ on $I$ such that
   $$\mathcal{R}_t^{Z}:=\Ric_t-h_t-\nabla^t Z_t\geq k(t),\quad t\in I.$$
   For any $\epsilon>0$, positive function $\ell$ on $I$ and $t\in I$,
   let
   \begin{align*}
     &A_1=2k-\ell,\quad B_1(t)=2d+\frac1{4}\l(3(d-1)\epsilon^{-1}+3k_{\epsilon}(t)\epsilon+2|Z_t|_t(o)\r)^2\ell^{-1}(t),
   \end{align*}
   where
   $$k_{\epsilon}(t)=\sup\left\{\l|\Ric_t\r|(x):\rho_t(x)\leq
     \epsilon\right\}.$$
   There exists a positive constant $\epsilon$ and a positive function
   $\ell$ on $I$ such that
   \begin{align}\label{eq4}
     H_1(t):=\int_{-\infty}^{t}\exp\l(-\int_r^{t}A_1(s)\,\vd s\r)B_1(r)\,\vd r<+\infty.
   \end{align}
 \end{enumerate}

\begin{remark}\label{rem2}
  In {\bf(H1)} and {\bf(H2)} the Lyapunov function
  $\varphi\circ \rho_t$ is by definition time-dependent.
  \begin{enumerate}[(a)]
  \item [(a)] From condition \eqref{eq4} it can be seen that the
    function $k$ in {\bf (H3)} must satisfy
    \begin{equation}\label{k-inter-esti}
      \int_{-\infty}^t\exp\l(-2\int_r^tk(s)\,\vd s\r)\,\vd r<\infty\ \text{ and }\ \int_{-\infty}^tk(s)\,\vd s=+\infty,\quad  t\in I.
    \end{equation}
  \item [(b)] Hypothesis {\bf(H1)} gives a sufficient condition for
    non-explosion of $L_t$-diffusions.  Hypothesis {\bf(H2)} ensures
    existence of an evolution system of measures $(\mu_t)_{t\in I}$,
    whereas {\bf(H3)} guarantees uniqueness of the evolution system of
    measures $(\mu_t)_{t\in I}$.
  \item [(c)] As indicated, the Lyapunov function
    $\varphi\circ \rho_t$ is time-dependent. Comparatively, in
    \cite{KLL10} the Euclidean distance is used as reference distance
    and then a space only Lyapunov condition is sufficient for
    existence and uniqueness of an evolution system of measures. In
    \cite{LLA13, KLL10} the coefficients in the Lyapunov condition are
    uniformly bounded, and as consequence a time-homogeneous process
    can be used for comparison with the original process. In our
    setting, the coefficients in the Lyapunov conditions need to be
    time-dependent to preserve the information about the varying
    space.
  \end{enumerate}
\end{remark}

% (where means that for any $r>0$ there exists $R>0$ such that
% $\mu_t(B_0(0,R))\geq 1-r$ for any $t\in I$).

In general, evolution systems of measures are far from being
unique. If there is a unique system it plays an important
role. Indeed, it is related to the asymptotic behavior of $P_{s,t}$ as
$s\rightarrow -\infty$. We shall prove that if Hypothesis {\bf(H3)}
holds, then for $x\in M$ and $(s,t)\in \Lambda$,
$$\lim_{s\rightarrow -\infty}\|P_{s,t}f(x)-\mu_t(f)\|_{L^2(M,\mu_s)}=0,$$
where $\mu_t(f)$ denotes the average of $f$ with respect to the
measure $\mu_t$.

In Sections \ref{Sect3}-\ref{Sect5} we use Hypothesis {\bf(H3)} as
standing assumption.  Taking the unique evolution system of measures
$(\mu_s)_{s\in I}$ as reference measures, we study contractivity
properties of the time-inhomogeneous semigroup $P_{s,t}$. For the sake
of brevity, we introduce the following notations:
\begin{align*}
  \|P_{s,t}\|_{(p,t)\rightarrow(q,s)}:=\|P_{s,t}\|_{L^p(M,\,\mu_t)\rightarrow L^q(M,\,\mu_s)},
  \quad \|P_{s,t}\|_{(p,t)\rightarrow \infty}:=\|P_{s,t}f\|_{L^p(M,\,\mu_t)\rightarrow L^{\infty}(M)}.
\end{align*}
\begin{definition}
  The evolution operator $P_{s,t}$ is called:
  \begin{enumerate}[(i)]
  \item \textit{hypercontractive} if it maps $L^p(M,\mu_t)$ into
    $L^q(M,\mu_s)$ for some $1<p<q<+\infty$ and $(s,t)\in \Lambda$
    such that
  $$\|P_{s,t}\|_{(p,t)\rightarrow(q,s)}\leq 1;$$
\item \textit{supercontractive} if it maps $L^p(M,\mu_t)$ into
  $L^q(M,\mu_s)$ for any $1<p<q<+\infty$ and $(s,t)\in \Lambda$, and
  if there exists a positive function
  $C_{p,q}: \Lambda\rightarrow (0,+\infty)$ such that
      $$\|P_{s,t}\|_{(p,t)\rightarrow(q,s)}\leq C_{p,q}(s, t);$$
    \item \textit{ultrabounded} if it maps $L^p(M,\mu_t)$ into
      $L^{\infty}(M)$ for every $p>1$ and $(s,t)\in \Lambda$, and if
      there exists a function
      $C_{p,\infty}: \Lambda\rightarrow (0,+\infty)$ such that
      \begin{align}\label{ultra}
        \|P_{s,t}f\|_{(p,t)\rightarrow \infty}\leq C_{p,\infty}(s, t).
      \end{align}
      % \item ``ultracontractive'' if it maps $L^p(M,\mu_s)$ into
      %   $L^{\infty}(M)$ for every $p\geq 1$, i.e. condition
      %   $\eqref{ultra}$ holds for every $p\geq 1$.
    \end{enumerate}

  \end{definition}
  \begin{remark}\label{rem4}
    It is easy to see that the function $C_{p,q}$ necessarily has the
    following properties due to contractivity of the semigroup:
    \begin{enumerate}[(i)]
    \item For fixed $s\in I$, the function
      $C_{p,q}(s,s+\newdot)\colon (0,\infty)\rightarrow(0,\infty)$ is
      a non-increasing function;
    \item for fixed $t\in I$, the function
      $C_{p,q}(t-\newdot,t)\colon (0,\infty)\rightarrow(0,\infty)$ is
      a non-increasing function.
    \end{enumerate}
    Note that the function $C_{p,q}$ takes into account both the
    position and the length of the interval $[s,t]$.
  \end{remark}

  In what follows, we use the abbreviation
$$\|\newdot\|_{p,\,s}:=\|\newdot\|_{L^p(M,\,\mu_s)}.$$
In Section \ref{Sect4}, we extend the arguments of \cite{RW03} to
consider hypercontractivity and supercontractivity via logarithmic
Sobolev inequalities (in short log-Sobolev inequalities).  In fact,
under the assumption that $\mathcal{R}_t^Z\geq k(t)$ for $t\in I$,
there is a family of log-Sobolev inequalities with respect to
$P_{s,t}$:
$$P_{s,t}(f^2\log f^2)\leq 4\l(\int_s^t\exp\l(-2\int_r^tk(u)\,\vd u\r)\,\vd r\r)P_{s,t}|\nabla^t f|_t^2+P_{s,t}f^2\log P_{s,t}f^2,\ \ f\in C_b^1(M),\ (s,t)\in \Lambda.$$
Hypercontractivity of $P_{s,t}$ in $L^p$ space, related to the unique
evolution system of measures, is then obtained as a consequence of the
log-Sobolev inequalities.

In Section \ref{Sect5} we then prove that supercontractivity of the
evolution operators $P_{s,t}$ is equivalent to the validity of the
following family of super-log-Sobolev inequalities
$$\int_M f^2\log \frac{|f|}{\|f\|_{2,s}}\,\vd \mu_s\leq r\big\||\nabla^s f|_s\big\|_{2,s}^2+\beta_s(r)\|f\|_{2,s}^2,\quad r >0,$$
for every $s\in I$, $f\in H^1(M,\mu_s)$ and some positive decreasing
function $\beta_s$.  Note that the function $\beta_s$ may depend on
the current time $s$ which generalizes the notion of super-log-Sobolev
inequalities for non-autonomous systems on $\R^d$ in \cite{LLA13}.
Moreover, combining the super-log-Sobolev inequalities and
dimension-free Harnack inequalities, we prove that the exponential
integrability of radial function with respect to $(\mu_t)_{t\in I}$ or
$(P_{s,t})_{(s,t)\in \Lambda}$ is equivalent to supercontractivity or
ultraboundedness of the corresponding semigroup.

The paper is organized as follows. In Section~\ref{Sect2} we first
give sufficient conditions for existence and uniqueness of evolution
systems of measures.  Then in Section~\ref{Sect3}, by means of Bismut
type formulas, gradient estimates in $L^{p}(M,\mu_s)$ are established
for $p\in (1,+\infty]$, which are used in
Sections~\ref{Sect4}-\ref{Sect5} to study hypercontractivity,
supercontractivity and ultraboundedness for the corresponding
semigroup.

\section{Diffusion processes and evolution system of
  measures}\label{Sect2}
\subsection{Non-explosion}
Recall that $\rho_t(x)$ denotes the distance function $\rho_t(o,x)$
with respect to a fixed reference point $o\in M$.  A sufficient
condition for non-explosion of $L_t$-diffusions can be given as
follows.
\begin{theorem}\label{nonep}
  Suppose that Hypothesis {\bf(H1)} holds. Then $L_t$-diffusion
  process $X_t$ is non-explosive before time~$T$.
\end{theorem}

\begin{proof}
  Without loss of generality, we suppose that the $L_t$-process $X_t$
  starts from $x$ at time $s$.  For fixed $t^*\in (s,T]$, there exists
  $c:=\sup_{t\in [s,t^*]}m(t)>0$ such that
$$\l(L_t+\partial_t\r)\varphi\circ \rho_t(x)\leq c,\quad (t,x)\in [s,t^*]\times M.$$
Then, by the It\^{o} formula for the radial part of $X_t$ (see
\cite[Theorem 2]{KP11}), we obtain
\begin{align*}
  \vd \varphi\circ\rho_t(X_t)&\leq
                               \sqrt2\l<u_t^{-1}\nabla^{t}\varphi\circ\rho_t(X_t),\vd
                               B_t\r>+\l(L_t+\partial_t\r)\varphi\circ \rho_t(X_t)\,\vd t\\
                             &\leq\sqrt2\l<u_t^{-1}\nabla^{t}\varphi\circ\rho_t(X_t),\vd
                               B_t\r>+c\,\vd t
\end{align*}
up to the lifetime $\zeta\wedge t^*$ where
$\zeta:=\lim_{n\rightarrow \infty} \zeta_n$ with
$$\zeta_n:=\inf\{t\in (s,T)\colon \rho_t(x,X_t)\geq n\}.$$
In particular, if $X_s=x\in M$, then
$$\varphi(n)\P^x\{\zeta_n\leq t\}\leq \E^x[\varphi\circ \rho_{t\wedge\zeta_n}(X_{t\wedge\zeta_n})]\leq \varphi(\rho_s(x))+ct,\quad  t\in [s,t^*].$$
According to Hypothesis {\bf(H1)}, $\varphi$ is an increasing function
such that $\varphi(r)\rightarrow +\infty$ as $r\rightarrow\infty$.
Thus, there exists $m\in \N^+$ such that $\varphi(n)>0$ for
all $n\geq m$ and
$$\P^x\{\zeta\leq t\}\leq \lim_{n\rightarrow\infty}\P^x\{\zeta_n\leq t\}
\leq\lim_{n\rightarrow\infty}\frac{\varphi(\rho_s(x))+ct}{\varphi(n)}=0,\quad
t\in[s,t^*].$$
Therefore we have $\P\{\zeta\geq t^*\}=1$. Since $t^*$ is
arbitrary, we obtain
$$\P\{\zeta\geq T\}=1$$
which completes the proof.
\end{proof}

From Theorem \ref{nonep} we get the following corollary which has been
proved in \cite{KP11} in the case of a Lyapunov condition with
constant coefficients.
\begin{corollary}\label{cor-2-1}
  Let $\psi \in C({\R}^+)$ and $h\in C(I)$ be non-negative such that
  for any $t\in I$,
  \begin{align}\label{LR}
    (L_t+\partial_t)\rho_t(x)\leq h(t)\psi(\rho_t(x))
  \end{align}
  holds outside $\mathrm{Cut}_{t}(o)$, the cut-locus of $o$ associated
  with the metric $g_t$. If
  \begin{align}\label{1e1}\int_1^{\infty}\vd
    t\int_1^{t}\exp{\l(-\int_r^t\psi(s)\,\vd s\r)}\vd
    r=\infty,\end{align}
  then the $L_t$-diffusion process
  is non-explosive.
\end{corollary}

\begin{proof}
  Suppose that the process $X_t$ generated by $L_t$ starts from $x$ at
  time $s\in I$. For fixed $t^*\in (s,T]$, let
  $c=\sup_{t\in [s,t^*]}h(t)$ and
$$\varphi(s)=\int_1^s\vd t\int_1^t \exp\l(-c\int_r^t\psi(s)\,\vd s\r)\vd r.$$
It is easy to see from condition \eqref{1e1} that $\varphi$ is an
increasing function on $\R^+$ with $\varphi(r)\rightarrow \infty$ as
$r\rightarrow \infty$, satisfying
\begin{align*}
  (L_t+\partial_t)\varphi(\rho_t(x))\leq 1, \quad\ t\in [s,t^*].
\end{align*}
This completes the proof.
\end{proof}

\subsection{Evolution systems of measures}
For $t\in I$ consider the linear second order differential operator
$L_t$ given on a smooth function $f$ by
$$L_tf=(\Delta_t +Z_t)f.$$
As indicated, Hypothesis {\bf(H1)} guarantees the existence of a
unique Markov semigroup $P_{s,t}$ generated by~$L_t$.
% and a unique operator $P_{s,t}$
Indeed, for fixed $t\in I$ and $f\in C_b(M)$, the function
$(s,x)\mapsto P_{s, t}f(x)$ is the unique bounded classical solution
in $C_b((-\infty, t]\times M)\cap C^{1,2}((-\infty,t]\times M)$ to the
backward Cauchy problem:
\begin{align}\label{heat-2}
  \begin{cases}
    \partial_su(\newdot, x)(s)=-L_su(s,\newdot)(x), & (s,x)\in (-\infty, t)\times M,\\
    u(t,x)=f(x),& x\in M.
  \end{cases}
\end{align}
According to the uniqueness of solutions to Eq.~\eqref{heat-2}, we
obtain
$$P_{s,r}\,P_{r,t}=P_{s,t},\quad s\leq r\leq t< T.$$
Moreover, for any $(s,t)\in\Lambda$, $x\in M$ and $f\in C^2(M)$ with
$\|L_tf\|_{\infty}^{\mathstrut}<\infty$, the forward Kolmogorov
equation reads as
$$\frac{\partial}{\partial t}P_{s,t}f(x)=P_{s,t}L_tf(x),$$
and for any $f\in \mathcal{B}_b(M)$, $(s,t)\in \Lambda$ and $x\in M$,
the backward Kolmogorov equation is given by
$$\frac{\partial }{\partial s}P_{s,t}f(x)=-L_sP_{s,t}f(x).$$

Based on Hypothesis {\bf(H2)} or {\bf(H3)}, one can prove existence
and uniqueness of an evolution system.

   \begin{theorem}\label{invariant-measure-th}
     Suppose that Hypothesis {\bf(H2)} holds, then there exists an
     evolution system of measures $(\mu_t)_{t\in I}$ for
     $(P_{s,t})_{(s,t)\in \Lambda}$ such that
     \begin{align}\label{ineq-e1}
       \sup_{s\in (-\infty,t]}\int_M (\varphi\circ \rho_s)(y)\,\mu_s(\vd y)\leq H(t).
     \end{align}
     Suppose that Hypothesis {\bf(H3)} holds, then there exists a
     unique evolution system and
     \begin{align}\label{ineq-1}
       \sup_{s\in (-\infty,t]}\int_M \rho_s(y)^2\,\mu_s(\vd y)< H_1(t).
     \end{align}
   \end{theorem}
   \begin{proof}
     (a) \ We first show existence.  Given $t\in I$, a family of
     measures can be constructed as follows (see e.g. \cite{BRW01b}
     for details). For $A\in \mathcal{B}(M)$ and $(s,t)\in \Lambda$,
     let
$$\mu_{s,t}(A):=\frac1{t-s}\int_s^tP_{r,t}(o,A)\,\vd r.$$
We claim that under Hypothesis {\bf(H2)}, the family of measures
$(\mu_{s,t})_{s\in (-\infty,t]}$ is compact.  Suppose that $X_t$
starts from $o$ at time $s$.  Under Hypothesis {\bf(H2)}, applying the
It\^{o} formula to the radial process $\rho_t(X_t)$, we get
\begin{align*}
  \vd \varphi(\rho_t(X_t))&\leq \varphi'(\rho_t(X_t))\l<\nabla^t\rho_t(X_t), u_t\vd B_t\r>_t+(L_t+\partial_t)\varphi\circ \rho_t(X_t)\,\vd t\\
                          &\leq \varphi'(\rho_t(X_t))\l<\nabla^t\rho_t(X_t), u_t\vd B_t\r>_t+\big(a(t)-c(t)\varphi\circ \rho_t(X_t)\big)\,\vd t.
\end{align*}
It follows that
\begin{align*}
  \E[\varphi(\rho_t(X_{t\wedge\zeta_n}))]-\varphi(0)\leq \E\int_s^{t\wedge\zeta_n}\big(a(r)-c(r)[\varphi\circ \rho_r(X_r)]
  \big)\,\vd r,
\end{align*}
i.e.,
\begin{align*}
  \E\l[\exp{\l(\int_s^{t\wedge \zeta_n}c(r)\,\vd r\r)}\varphi\circ \rho_{t\wedge \zeta_n}(X_{t\wedge \zeta_n})\r]\leq
  \varphi(0)+\E\left[\int_s^{t\wedge\zeta_n}\exp{\l(\int_s^rc(u)\,\vd u\r)}a(r)\,\vd r\right].
\end{align*}
Using the condition $\varphi(0)=0$ and letting $n\rightarrow \infty$,
we arrive at
\begin{align*}
  \E\left[\varphi\circ \rho_{t}(X_{t})\right]&\leq
                                               \l[\int_s^{t}\exp{\l(\int_s^rc(u)\,\vd u\r)}a(r)\,\vd r\r]\exp{\l(-\int_s^{t}c(r)\,\vd r\r)}\\
                                             &\leq \int_s^{t}\exp{\l(-\int_r^tc(u)\,\vd u\r)}a(r)\,\vd r\\
                                             &\leq H(t).
\end{align*}
Therefore, according to the monotonicity of $H$, we have
\begin{align}\label{P-bound}
  \sup_{s\in (-\infty,t]}(P_{s,t}\varphi\circ \rho_t)(o)\leq H(t),
\end{align}
from which it follows that
\begin{align*}
  \mu_{s,t}(\varphi\circ \rho_t)&=\frac1{t-s}\int_s^t(P_{r,t}\varphi\circ \rho_t)(o)\,\vd r\leq H(t).
\end{align*}
In addition, since $\varphi$ is a compact and increasing function such
that $\varphi(r)\rightarrow +\infty$ as $r\rightarrow +\infty$, we
know that $(\mu_{s,t})_{s\in (-\infty,t]}$ is a family of compact
measures, i.e., for each $n\in\Z$, there exists a sequence
$(t_{n_k})$, $t_{n_k}\rightarrow +\infty$ as $k\rightarrow +\infty$
such that
$$\mu_{t_{n_k},n}\rightharpoonup^* \mu_n.$$
Define $\mu_s:=P_{n,s}^*\mu_n$. It is easy to check that the family
$\mu_s$ satisfies Eq.~\eqref{invariant-measure}, i.e., for
$\phi\in \mathscr{B}_b(M)$,
$$\mu_s(P_{s,t}\phi)=P_{n,s}^*\mu_n(P_{s,t}\phi)=\mu_n(P_{n,t}\phi)=\mu_t(\phi).$$
By this and the bound \eqref{P-bound}, we get the existence of an
evolution system $(\mu_{s})_{s\in I}$. Moreover, we have the estimate
\begin{align*}
  \mu_s(\varphi\circ \rho_s)=\mu_n(P_{n,s}\varphi\circ \rho_s)\leq \lim_{t_{n_k}\rightarrow \infty} \frac1{n-t_{n_k}}\int_{t_{n_k}}^n\sup_{r\in (-\infty,s]} (P_{r,s}\varphi\circ \rho_s)(o)\,\vd r\leq H(s)
\end{align*}
which completes the proof of Eq.~\eqref{ineq-e1}.

(b) \ If Hypothesis {\bf(H3)} holds, we claim that there exists a
unique evolution system of probability measures $(\mu_t)_{t\in I}$
such that
$$\sup_{s\in(-\infty,t]}\mu_s(\rho_s^2)<H_1(t).$$
First recall the formula (see \cite[Lemma 5 and Remark 6]{MT})
\begin{equation}\label{par-rho}
  \partial_t\rho_t(x)=\frac12\int\partial_tg_t(\dot{r}(s),\dot{r}(s))\,\vd s
\end{equation}
where $r\colon[0,\rho_t(x)]\rightarrow M$ is a $g_t$-geodesic
connecting $o$ and $x$. By this formula and the index lemma, we have
\begin{align}\label{L-rho}
  (&L_t+\partial_t)\rho_t=(\Delta_t+Z_t+\partial_t)\rho_t\nonumber\\
   &\leq\frac{(d-1)G'(\rho_t)}{G(\rho_t)}+\int_0^{\rho_t}\frac12\partial_tg_t(\dot{r}(s),\dot{r}(s))\,\vd s+
     \int_0^{\rho_t}(\nabla^t Z_t)(\dot{r}(s),\dot{r}(s))\,\vd s+\l<Z_t, \dot{r}(0)\r>_t(o)
\end{align}
where $G$ is the solution to the equation
$$\left\{
\begin{aligned}
  G''(s)&=\frac{-\Ric_t(\dot{r}(s),\dot{r}(s))}{d-1}\,G(s),\\
  G(0)&=0,\ G'(0)=1.
\end{aligned}
\right.$$ Under Hypothesis {\bf(H3)}, by \cite[Lemma 9]{KP11}, we have
\begin{align*}
  (L_t+\partial_t)\rho_t&\leq \frac{(d-1)G'(\rho_t)}{G(\rho_t)}-k(t)\rho_t+\int_0^{\rho_t}\Ric_t(\dot{r}(s),\dot{r}(s))\,\vd s+\l<Z_t, \dot{r}(0)\r>_t(o)\\
                        &\leq \frac{(d-1)G'(\rho_t)}{G(\rho_t)}-k(t)\rho_t-\int_0^{\rho_t}\frac{(d-1)G''(s)}{G(s)}\,\vd s+|Z_t|_t(o)\\
                        &\leq F_t(\rho_t)-k(t)\rho_t+|Z_t|_t(o)
\end{align*}
where
$F_t(s)=\sqrt{k_{\epsilon}(t)(d-1)}\coth\left(\sqrt{k_{\epsilon}(t)/(d-1)}(s\wedge
  \epsilon)\right)+k_{\epsilon}(t)(s\wedge \epsilon)$ and
 $$k_{\epsilon}(t):=\sup\{|\Ric_t|: \rho_t(x)\leq \epsilon\}.$$
 It is easy to see that $F_t(s)$ is non-increasing in $s$ and
 $\lim_{r\rightarrow 0}r F_t(r)<\infty$.  Hence, by means of the
 positive function $\ell$ in Hypothesis {\bf(H3)}, we obtain
 \begin{align*}
   (L_t+\partial_t)\rho_t^2&=2\rho_t(L_t+\partial_t)\rho_t+2\\
                           &\leq 2\rho_t (F_t(\rho_t)-k(t)\rho_t+|Z_t|_t(o))+2\\
                           &\leq 2d+2\l\{k_{\epsilon}(t)\epsilon+{(d-1)}{\epsilon}^{-1}+\sqrt{(d-1)k_{\epsilon}(t)}+|Z_t|_t(o)\r\}\rho_t-2k(t) \rho_t^2\\
                           &\leq 2d+\l\{3\l(k_{\epsilon}(t)\epsilon+{(d-1)}{\epsilon}^{-1}\r)+2|Z_t|_t(o)\r\}\rho_t-2k(t) \rho_t^2\\
                           &\leq 2d+\frac{\l\{3k_{\epsilon}(t)\epsilon+3(d-1)\epsilon^{-1}+2|Z_t|_t(o)\r\}^2}{4\ell(t)}-(2k(t)-\ell(t))\rho_t^2.
 \end{align*}
 By a similar argument as in part (a), we obtain an evolution system
 of measures such that
$$\sup_{t\in (-\infty, s]}\mu_t(\rho_t^2)\leq H_1(s).$$

We now use a coupling method to prove uniqueness of the evolution
system. Let $(X_t, Y_t)$ be a parallel coupling starting from $(x,y)$
at time $s$.  Then, by \cite{Cheng15} or \cite{Ku}, we know that if
$\mathcal{R}_t^{Z}\geq k(t)$, $t\in I$, then
$$\E^{(s,(x,y))}[\rho_t(X_t,Y_t)]\leq \exp\l(-\int_s^tk(r)\,\vd r\r)\rho_s(x,y).$$
Let $(\mu_t)_{t\in I}$ be an evolution system of measures. Then, we
have the estimate:
\begin{align}\label{esti-BV}
  \l|P_{s,t}f(o)-\mu_t(f)\r|&=\l|\int (P_{s,t}f(o)-P_{s,t}f(y))\,\mu_s(\vd y)\r|\notag\\
                            &=\l|\int \E^{(s,(o,y))}\l[\frac{f(X_t)-f(Y_t)}{\rho_t(X_t,Y_t)}\rho_t(X_t,Y_t)\r]\mu_s(\vd y)\r|\notag\\
                            &\leq \||\nabla^t f|_t\|_{\infty}^{\mathstrut}\int \E^{(s,(o,y))}\l[\rho_t(X_t,Y_t)\r]\,\mu_s(\vd y)\notag\\
                            &\leq \exp\l(-\int_s^tk(r)\,\vd r\r)\||\nabla^t f|_t\|_{\infty}^{\mathstrut}\,\mu_s(\rho_s)\notag\\
                            &\leq \exp\l(-\int_s^tk(r)\,\vd r\r)\||\nabla^t f|_t\|_{\infty}^{\mathstrut}\big(\mu_s(\rho_s^2)\big)^{1/2}.
\end{align}
In addition, from Eq. \eqref{k-inter-esti}, we know that
$$\exp\l(-\int_{-\infty}^tk(r)\,\vd r\r)=0\ \  \mbox{ and}\ \ \sup_{s\in (-\infty, t]}\mu_s(\rho_s^2)<\infty.$$ Now letting
$s\rightarrow -\infty$, we conclude that
 $$\lim _{s\rightarrow -\infty}\left|P_{s,t}f(o)-\mu_t(f)\right|=0.$$
 If there exists another evolution system of probability measures
 $(\nu_t)_{t\in I}$, then $\nu_t(f)$ is also the limit of $P_{s,t}f(o)$ as
 $s\rightarrow -\infty$, and hence $\nu_t=\mu_t$.
\end{proof}

Directly from Eq. \eqref{esti-BV} we have the following asymptotic
results.

\begin{corollary}\label{cor1}
  Suppose that Hypothesis {\bf(H3)} holds. Then we have the following
  convergence result: for any $f\in C^1(M)$ being constant outside a
  compact set, there exists a function $c$ in $C(I)$ such that
$$\|P_{s,t}f-\mu_t(f)\|_{2,s}\leq c(t)\exp\l(-\int_s^tk(r)\,\vd r\r)\||\nabla^tf|_t\|_{\infty}^{\mathstrut},\quad (s,t)\in \Lambda.$$
\end{corollary}
\begin{proof}
  Let $(X_t,Y_t)$ be parallel coupling process associated to $L_t$.
  For any $f\in C^1(M)$ being constant outside a compact set, we have
  \begin{align}\label{ineq-2}
    |P_{s,t}f(x)-\mu_t(f)|&=\left|P_{s,t}f(x)-P_{s,t}f(o)+P_{s,t}f(o)-\mu_t(f)\right|\notag\\
                          &\leq \l|\E^{(s,(x,o))}\l[f(X_t)-f(Y_t)\r]\r|+\e^{-\int_s^tk(r)\,\vd r}\||\nabla^t f|_t\|_{\infty}^{\mathstrut}\l(\mu_s(\rho_s^2)\r)^{1/2}\notag\\
                          &\leq \l|\E^{(s,(x,o))}\l[\frac{f(X_t)-f(Y_t)}{\rho_t(X_t,Y_t)}\rho_t(X_t,Y_t)\r]\r|+\e^{-\int_s^tk(r)\,\vd r}\||\nabla^t f|_t\|_{\infty}^{\mathstrut}\l(\mu_s(\rho_s^2)\r)^{1/2}\notag\\
                          & \leq \||\nabla^tf|_t\|_{\infty}^{\mathstrut}\E^{(s,(x,o))}\l[\rho_t(X_t,Y_t)\r]+\e^{-\int_s^tk(r)\,\vd r}\||\nabla^t f|_t\|_{\infty}^{\mathstrut}\l(\mu_s(\rho_s^2)\r)^{1/2}\notag\\
                          &\leq \e^{-\int_s^tk(r)\,\vd r}\||\nabla^t f|_t\|_{\infty}^{\mathstrut}\l(\rho_s(x)+\l(\mu_s(\rho_s^2)\r)^{1/2}\r)
  \end{align}
  which implies that
$$\|P_{s,t}f-\mu_t (f)\|_{2,s}\leq 2\exp\l(-\int_s^tk(r)\,\vd r\r)\||\nabla^t f|_t\|_{\infty}^{\mathstrut}\l(\mu_s(\rho_s^2)\r)^{1/2}.$$
Now using Theorem \ref{invariant-measure-th} and
$$\sup_{s\in (-\infty, t]}\mu_s(\rho_s^2)<\infty,$$ 
we obtain the result directly.
\end{proof}

\begin{corollary}\label{cor2} Suppose that Hypothesis {\bf(H3)} holds and $\sup_{s\in (-\infty,t]}\rho_s(x)<\infty$ for any $x\in M$ and $t\in I$. Then we have the following convergence
  result: for any $f\in C_b^1(M)$, there exists a function $C$ in
  $C(I)$ such that
$$|P_{s,t}f-\mu_t(f)|\leq C(t)\exp\l(-\int_s^tk(r)\,\vd r\r)\||\nabla^tf|_t\|_{\infty}^{\mathstrut},\quad (s,t)\in \Lambda.$$
\end{corollary}

\begin{proof}
  If $\sup_{s\in (-\infty,t]}\rho_s(x)<\infty$ for any $x\in M$ and
  $t\in I$, then the result can be directly derived from the
  inequality \eqref{ineq-2}.
\end{proof}

\begin{remark}\label{heat-back}
  Actually, our results can be applied to the following forward Cauchy
  problem via a time reversal: for $s\in [T,+\infty)$,
$$
\begin{cases}
  \partial_tu(\newdot,x)(t)=L_tu(t,\newdot)(x),\ \ & (t,x)\in (s,+\infty)\times M;\\
  u(s,x)=f(x),\ \ &\ x\in M.
\end{cases}
$$
\end{remark}

\section{Gradient estimates}\label{Sect3}

We now turn to gradient estimates for the semigroup. It is well known
that the so-called Bismut formula is a powerful tool to derive
gradient estimates of semigroups in the fixed metric case (see
\cite{bismut, ElworthyLiforms}).  Let us first recall a Bismut type
formula for $\nabla^s P_{s,t}f$ (see \cite[Corollary\ 3.2]{Cheng15}). 
To this end, define an
$\R^d\otimes \R^d$-valued process
$(Q_{s,t})_{(s,t)\in \Lambda}$ as the solution to the following
ordinary differential equation
\begin{align}\label{damp}
  \frac{\vd Q_{s,t}}{\vd
  t}=-\mathcal{R}^Z_t(u_t)Q_{s,t},\quad Q_{s,s}=\id,\ (s,t)\in \Lambda, \end{align}
where $u_t$ is the horizontal
$L_t$-diffusion process $X_{t}^{(s,x)}$ with $\pi(u_s)=x$, and
$\mathcal{R}_t^{Z}(u_t)\in \R^d\otimes\R^d$ satisfies
\begin{align*}
  \l<\mathcal{R}^{Z}_t(u_t)a,b\r>_{\R^d}&=\mathcal{R}_t^{Z}(u_ta,
                                                  u_t b),\quad a,b\in \R^d.
\end{align*}
If $\mathcal{R}_t^{Z}\geq k(t)$, $t\in I$ then we have
\begin{align}\label{add-Q}
  \|Q_{r,t}\|\leq \exp{\l(-\int_r^tk(s)\,\vd s\r)},\quad (r,t)\in \Lambda,
\end{align}
where $\|\newdot\|$ is the operator norm on $\R^d$. The
following is the derivative formula taken from \cite{Cheng15}.

\begin{proposition}\label{cheng}
  Assume that $\mathcal{R}_t^Z\geq k(t)$ for some continuous function
  $k$ on $I$. Let $(s,t)\in \Lambda$.  Then for $f\in C^1(M)$ such
  that $f$ is constant outside a compact set, and for any
  $h\in C_b^1([s,t])$ satisfying $h(s)=0$ and $h(t)=1$, we have
  \begin{align}\label{eq-4}
    u_s^{-1}\nabla^sP_{s,t}f(x)=\E^{(s,x)}\l[Q_{s,t}^*u_{t}^{-1}\nabla^tf(X_{t})\r]=\frac1{\sqrt2}\E^{(s,x)}
    \l[f(X_{t})\int_s^t h'(r)Q_{s,r}^*\vd B_r\r]
  \end{align}
  where $Q_{s,t}^*$ is the transpose of $Q_{s,t}$.
\end{proposition}

This pointwise gradient estimate can be derived from Proposition
\ref{cheng}.
\begin{theorem}\label{cor-gradient}
  Suppose that Hypothesis {\bf(H3)} holds. Let $(\mu_t)_{t\in I}$ be
  the evolution system of measures for $P_{s,t}$. Then,
  \begin{enumerate}[\rm(a)]
  \item for every $f\in C^1(M)$ such that $f$ is constant outside a
    compact set and $1\leq p<\infty$,
    \begin{equation}\label{grad-1}
      \l\||\nabla^s P_{s,t}f|_s\r\|_{p, s}\leq \exp\l(-\int_s^tk(r)\,\vd r\r)\l\||\nabla^t f|_t\r\|_{p,t},\quad 
      (s,t)\in \Lambda;
    \end{equation}
  \item for any $1<p<\infty$, there exists a positive constant
    $C_1=C_1(p)$ such that
  $$\l\||\nabla^s P_{s,t}f|_s\r\|_{p, s}\leq C_1\l(\max_{r\in [s,(t-1)\vee s]}
  \int_{r}^{(r+1)\wedge t}\exp\l(2\int_s^rk(u)\,\vd u\r)\vd
  r\r)^{-1/2}\|f\|_{p,t},\quad (s,t)\in \Lambda$$
  holds for every $f\in \mathscr{B}_b(M)$ and $x\in M;$
\item for $f\in \mathscr{B}_b(M)$, there exists a positive constant
  $C_1=C_1(p)$ such that
    $$\l\||\nabla^sP_{s,t}f|_s\r\|_{\infty}^{\mathstrut}\leq C_1\l(\max_{r\in [s,(t-1)\vee s]}\int_{r}^{(r+1)\wedge t}\exp\l(2\int_s^rk(u)\,\vd u\r)\vd r\r)^{-1/2}\|f\|_{\infty}^{\mathstrut},\quad(s,t)\in \Lambda.$$
  \end{enumerate}
\end{theorem}

\begin{proof}
  By the first equality in \eqref{eq-4} and inequality~\eqref{add-Q},
  the first assertion in (a) can be derived directly.  It is also easy
  to see that (c) follows from (b). Hence, it suffices to prove (b).

  For $p\in (1,\infty)$ and $t-s\leq 1$, by using the integration by
  parts formula, we have
  \begin{align}\label{eq:Gradest}
    |\nabla^sP_{s,t}f|_s^p(x)&=\frac1{\sqrt2}\Big|\E^{(s,x)}
                               \Big[f(X_{t})\int_s^t h'(r)Q_{s,r}^*\vd B_r\Big]\Big|^p \notag\\
                             &\leq \frac1{\sqrt2}P_{s,t}|f|^p(x)\l(\E^{(s,x)}\Big|\int_s^t h'(r)Q_{s,r}^*\vd B_r\Big|^q\r)^{p/q}\notag\\
                             &\leq \frac{c_p^p}{\sqrt2}P_{s,t}|f|^p(x)\l(\E^{(s,x)}\Big|\int_s^t {h'}^2(r)\|Q_{s,r}\|^2\,\vd r\Big|^{q/2}\r)^{p/q}\notag\\
                             &\leq \frac{c_p^p}{\sqrt2}P_{s,t}|f|^p(x)\l(\E^{(s,x)}\Big|\int_s^t {h'}^2(r)
                               \exp\l(2\int_s^rk(u)\,\vd u\r)\,\vd r\Big|^{q/2}\r)^{p/q}
  \end{align}
  where
  $$h(r)=\frac{\int_s^r\exp\l(2\int_s^{\rho}k(u)\,\vd u\r)\,\vd\rho}
  {\int_s^t\exp\l(2\int_s^{\rho}k(u)\,\vd u\r)\,\vd\rho}.$$
  It then follows that
  \begin{align*}
    |\nabla^sP_{s,t}f|_s^p(x)\leq \frac{c_p^p}{\sqrt2}P_{s,t}|f|^p(x)
    \l(\int_s^t\exp\l(2\int_s^{r}k(u)\,\vd u\r)\,\vd r\r)^{-p/2}.
  \end{align*}
  Integrating both sides of the inequality above with respect to
  $\mu_s$, we arrive at
  \begin{align}\label{grad-2}
    \mu_s(|\nabla^sP_{s,t}f|_s^p)\leq \frac{c_p^p}{\sqrt2}\mu_t(|f|^p)
    \l(\int_s^t\exp\l(2\int_s^{r}k(u)\,\vd u\r)\,\vd r\r)^{-p/2}.
  \end{align}

  It leaves us to check the case for $t-s>1$. For any $r\in [s,t-1]$, 
combining Eq.~\eqref{grad-1} and Eq.~\eqref{grad-2}, we have
  \begin{align*}
 & |\nabla^sP_{s,r}P_{r,t}f(x)|^p_s\leq \exp\l(-p\int_s^{r}k(r)\,\vd r\r)P_{s,r}|\nabla^{r}P_{r,t}f|_{r}^p(x)\\
  &\ \leq \frac{c_p^p}{\sqrt{2}}\exp\l(-p\int_s^{r}k(r)\,\vd r\r)
 \l(\int_r^{r+1}\exp\l(\int_r^{\rho}k(u)\,\vd u\r)\,\vd \rho\r)^{-p/2}P_{s,r}(P_{r,r+1}|P_{r+1,t}f|^p)(x)\\
  &\ \leq \frac{c_p^p}{\sqrt{2}}\l(\int_r^{r+1}\exp\l(\int_s^{\rho}k(u)\,\vd u\r)\,\vd \rho\r)^{-p/2}P_{s,t}|f|^p(x).
  \end{align*}
Integrating both sides by $\mu_s$ and minimizing the coefficient in $r$,  we obtain the
desired conclusion.
\end{proof}
\begin{remark}\label{rem3}
  In Theorem \ref{cor-gradient}(c), the inequality does not need some evolution
  system of measures as the reference measures. So the condition for
  this result can be weaken by only using
$$\mathcal{R}_t^Z\geq k(t)$$
for some function $k\in C(I)$.
\end{remark}

\section{Log-Sobolev inequality and hypercontractivity}\label{Sect4}
In this section, we prove hypercontractivity for $P_{s,t}$. Let us
first introduce the following log-Sobolev inequality, which is
essential to the proof of our hypercontractivity theorem.

\begin{proposition}\label{P-log-S}
  If $\mathcal{R}_t^Z\geq k(t)$ for some function $k\in C(I)$ then for
  any $p\in (1,\infty)$,
  \begin{align}\label{Log-S-P}
    P_{s,t}(f^2\log f^2)\leq 4\l(\int_s^t\exp\l(-2\int_r^tk(u)\,\vd u\r)\vd r\r)P_{s,t}|\nabla^t f|_t^2+P_{s,t}f^2\log P_{s,t}f^2,\quad (s,t)\in \Lambda,
  \end{align}
  holds for $f \in C_c^1(M)$.
\end{proposition}
\begin{proof}
  Without loss of generality, we suppose $f>\delta>0$. Otherwise, let
  $f_{\delta}=(f^2+\delta)^{{1}/{2}}$. Then by letting
  $\delta \rightarrow 0$, we obtain the conclusion.

  Consider the process
  $(P_{r,t}f^2)\log (P_{r,t}f^2)(X_{r\wedge \tau_n})$ where as above
  \begin{align}\label{tau-n}
    \tau_n=\inf\{t\in (s,T]:\rho_t(X_t)\geq n\}, \quad n\geq 1.
  \end{align}
  Applying It\^o's formula, we have
  \begin{align*}
    \vd (P_{r,t}f^2)\log (P_{r,t}f^2)(X_r)&=\vd M_r+(L_r+\partial_r)(P_{r,t}f^2\log P_{r,t}f^2)(X_r)\,\vd r\\
                                          &=\vd M_r + \l(\frac1{P_{r,t}f^2}|\nabla^rP_{r,t}f^2|_r^2\r)(X_r)\,\vd r,\quad s<r<\tau_n\wedge t,
  \end{align*}
  where $M_r$ is a local martingale.  By this and the estimate,
$$|\nabla^r P_{r,t}f^2|^2_r\leq \exp\l(-2\int_r^tk(u)\,\vd u\r)\l(P_{r,t}|\nabla^tf^2|_t\r)^2
\leq 4\exp\l(-2\int_r^tk(u)\,\vd u\r)(P_{r,t}f^2)P_{r,t}|\nabla^t f|_t^2,$$
we obtain
$$\vd (P_{r,t}f^2)\log (P_{r,t}f^2)(X_r)\leq \vd M_r+4\exp\l(-2\int_r^tk(u)\,\vd u\r)P_{r,t}|\nabla^tf|_t^2(X_r)\,\vd r,\quad   s<r<\tau_n\wedge t.$$
Integrating both sides from $s$ to $t\wedge \tau_n$, we have
\begin{align*}
  &\E^{(s,x)}\big[f^2\log f^2(X_{t\wedge \tau_n})\big]\\
                         &\leq \l(P_{s,t}f^2\log P_{s,t}f^2\r)(x)+\E^{(s,x)}\l[\int_s^{t\wedge \tau_n}4\exp\l(-2\int_r^tk(u)\,\vd u\r)P_{r,t}|\nabla^tf|_t^2(X_r)\,\vd r\r].
\end{align*}
Then again by dominated convergence, letting $n \uparrow +\infty$, we obtain
\begin{equation*}
  P_{s,t}(f^2\log f^2)\leq 4\l(\int_s^t\exp\l(-2\int_r^tk(u)\,\vd u\r)\,\vd r\r)P_{s,t}|\nabla^t f|_t^2+P_{s,t}f^2\log P_{s,t}f^2.\qedhere
\end{equation*}
\end{proof}

% \begin{corollary}\label{log-S}
%   Suppose that Hypothesis 2.2 holds. Then there holds
%$$\mu_t(f^2\log f^2)\leq \frac{4}{k}\mu_t(|\nabla^t f|_t^2)+\mu_t(f^2)\log \mu_t(f^2).$$
% \end{corollary}
The log-Sobolev inequality leads to the hypercontractivity of
$(P_{s,t})$.

\begin{theorem}
  Suppose that Hypothesis {\bf(H3)} holds and $(\mu_t)$ is the
  evolution system of measures for $P_{s,t}$.  Let $(s,t)\in \Lambda$
  and $p,q\in (1,\infty)$ such that
$$q\leq \exp{\l(-\frac1{4}\int_s^t\l(\int_r^t\exp\l(-2\int_u^tk(z)\,\vd z\r)\,\vd u\r)^{-1}\,\vd r\r)}(p-1)+1.$$
Then $P_{s,t}: L^p(M,\mu_t)\rightarrow L^q(M,\mu_s)$ satisfies
$$\|P_{s,t}\|_{(p,t)\rightarrow (q,s)}\leq 1.$$
\end{theorem}

\begin{proof}
  For the sake of conciseness, we assume $f>\delta>0$, otherwise we
  can use a similar argument as in the proof of Proposition
  \ref{P-log-S}. Consider the process
  $(P_{s,t}f)^{q(s)}(X_{s\wedge \tau_n})$, where
 $$q(s)=\exp{\l(-\frac1{4}\int_s^t\l(\int_r^t\e^{-2\int_u^tk(z)\,\vd z}\,\vd u\r)^{-1}\,\vd r\r)}(p-1)+1.$$
  Using the It\^{o} formula, we have that for $ s<\tau_n\wedge t$,
\begin{align*}
  \vd (P_{s,t}f)^{q(s)}(X_s)&=\vd M_s+(L_s+\partial_s)(P_{s,t}f)^{q(s)}(X_s)\,\vd s\\
                            &=\vd M_s+(P_{s,t}f)^{q(s)}\l[q(s)(q(s)-1)|\nabla^s\log P_{s,t}f|_s^2+q'(s)\log P_{s,t}f\r](X_s)\,\vd s.
\end{align*}
Therefore, for $r\leq s\leq t<T$,
\begin{align*}
  \E^{(r,x)}&\big[(P_{s,t}f)^{q(s)}(X_{s\wedge \tau_n})\big]-(P_{r,t}f)^{q(r)}(x)\\
  =&\int_r^{s}\bigg(q(u)(q(u)-1)\E^{(r,x)}\l[(P_{u,t}f)^{q(u)-2}|\nabla^uP_{u,t}f|_u^2(X_{u\wedge \tau_n})\r]\\
                                      &\qquad \quad \quad \ +q'(u)\E^{(r,x)}\l[(P_{u,t}f)^{q(u)}\log P_{u,t}f(X_{u\wedge \tau_n})\r]\bigg)\,\vd u.
\end{align*}
By using the dominant convergence theorem and letting
$n\rightarrow +\infty$, we have
\begin{align*}
  P_{r,s}(P_{s,t}f&)^{q(s)}(x)-(P_{r,t}f)^{q(r)}(x)\\
                  &= \int_r^s\bigg[q(u)(1-q(u))P_{r,u}\l((P_{u,t}f)^{q(u)-2}|\nabla^uP_{u,t}f|_u^2\r)(x)\\
                  &\quad \quad \qquad  +q'(u)P_{r,u}((P_{u,t}f)^{q(u)}\log P_{u,t}f)(x)\bigg]\,\vd u
\end{align*}
which implies
\begin{align*}
  \frac{\vd}{\,\vd s}P_{r,s}(P_{s,t}f)^{q(s)}&=q'(s)P_{r,s}((P_{s,t}f)^{q(s)}\log P_{s,t}f)\\
                                           &\quad+q(s)(q(s)-1)P_{r,s}((P_{s,t}f)^{q(s)-2}|\nabla^sP_{s,t}f|_s^2).
\end{align*}
Therefore, for $(P_{r,s}(P_{s,t}f)^{q(s)})^{1/q(s)}$, we have
\begin{align*}
  &\frac{\vd }{\,\vd s}(P_{r,s} (P_{s,t}f)^{q(s)})^{1/q(s)}\\
                            &=(P_{r,s}(P_{s,t}f)^{q(s)})^{1/q(s)}\l(-\frac{q'(s)}{q(s)^2}\log P_{r,s}(P_{s,t}f)^{q(s)}+\frac1{q(s)}\frac{\partial_s(P_{r,s}(P_{s,t}f)^{q(s)})}{P_{r,s}(P_{s,t}f)^{q(s)}}\r)\\
                            &\leq (P_{r,s}(P_{s,t}f)^{q(s)})^{\frac{1-q(s)}{q(s)}}\l[\l(4\int_s^t\exp\l(-2\int_r^tk(u)\,\vd u\r)\,\vd r\r)q'(s)-q(s)+1\r]\\
&\quad\times P_{r,s}\left((P_{s,t}f)^{q(s)-2}|\nabla^sP_{s,t}f|_s^2\right)
\end{align*}
where the last inequality comes from the log-Sobolev inequality
\eqref{Log-S-P}.  According to the definition of $q(s)$, we have
$$\frac{\vd}{\,\vd s}\l(P_{r,s}(P_{s,t}f)^{q(s)}\r)^{1/{q(s)}}\leq 0.$$
Integrating both sides from $s$ to $t$, we obtain
$$\l(P_{r,s}(P_{s,t}f)^{q(s)}\r)^{1/{q(s)}}\leq (P_{r,t}f^p)^{1/p}.$$
From this and the fact that  $q(s)/p\leq 1$, it follows that
$$\mu_r(P_{r,s}(P_{s,t}f)^{q(s)})\leq \mu_r(P_{r,t}f^p)^{q(s)/p}\leq (\mu_r(P_{r,t}f^p))^{q(s)/p},$$
which implies
$$\|P_{s,t}f\|_{q(s), s}\leq \|f\|_{p,t}.$$
This completes the proof.
\end{proof}

\section{Supercontractivity and ultraboundedness}\label{Sect5}
This section is devoted to supercontractivity and ultraboundedness for
the semigroup $P_{s,t}$ under Hypothesis {\bf(H3)}.
\subsection{Super log-Sobolev inequality and boundedness of semigroup}
We present a supercontractivity result first.
\begin{theorem}\label{log-S-superbound}
  Suppose that Hypothesis {\bf(H3)} holds. Let $(\mu_t)$ be the
  evolution system of measures associated with $P_{s,t}$.  Then the
  following properties are equivalent.
  \begin{enumerate}[{\rm (a)}]
  \item The semigroup $P_{s,t}$ is supercontractive.
  \item The family of super-log-Sobolev inequalities
    \begin{align}\label{log-S-I}
      \int f^2\log \frac{f^2}{\|f\|_{2,s}^2}\vd \mu_s\leq r \l\||\nabla^sf|_s\r\|_{2,s}^2+\beta_s(r)\|f\|^2_{2,s},\quad r>0,
    \end{align}
    holds for every $f\in H^1(M,\mu_s)$, $s\in I$, and some positive
    non-increasing function
    $\beta_s\colon (0,+\infty)\rightarrow (0,+\infty).$
  \end{enumerate}
\end{theorem}

First, we give a lemma which makes the proof of this theorem more concise.

\begin{lemma}\label{lem1}
  Suppose Hypothesis {\bf(H2)} holds. Let $(\mu_t)$ be an evolution
  system of measures for $P_{s,t}$. If
  $f\in C^{1,2}(I\times M)\cap C(I,L^1(M,\mu_r))$ and there exists
  some function $g \in \mathcal{B}_b(M)$ such that
  $|(\partial_r+L_r)f|\leq g$ for all $r\in I$, then
  \begin{align}\label{eq-3}
    \frac{\vd}{\vd r} \int f(r,x)\mu_r(\vd x)=\int (\partial_r+L_r)f(r,x)\mu_r(\vd x)
  \end{align}
  for every $r\in I$.
\end{lemma}

\begin{proof}
  For $f\in C^{1,2}(I\times M)\cap C(I, L^1(M,\mu_r))$, we have
$$ \int f(r,x)\mu_r(\vd x)=\int P_{s,r}f(r,x)\mu_s(\vd x),\quad s<r\leq T.$$
On the other hand,  using Kolmogorov's formula, we have
 $$\frac{\vd }{\vd r}P_{s,r}f(r,x)=P_{s,r}(L_r+\partial_r)f(r,x).$$
We complete the proof by applying the dominated convergence theorem.
\end{proof}

\begin{proof}[Proof of Theorem \ref{log-S-superbound}]
  First we prove ``$\text{(b)}\Rightarrow\text{(a)}$''. Let $(s,t)\in \Lambda$
  and $f\in C_c^{\infty}(M)$ such that $f>\delta>0$. By Lemma
  \ref{lem1}, we need to check the following to handle the
  derivative of $ \mu_s(P_{s,t}f)^{q(s)}$ with respect to $s$:
  \begin{align*}
    (L_s+&\partial_s)(P_{s,t}f)^{q(s)}\\
         &=L_s(P_{s,t}f)^{q(s)}-q(s)(P_{s,t}f)^{q(s)-1}(L_sP_{s,t}f)+q'(s)(P_{s,t}f)^{q(s)}\log P_{s,t} f\\
         &=q(s)(q(s)-1)|\nabla^sP_{s,t}f|_s^2(P_{s,t}f)^{q(s)-2}+q'(s)(P_{s,t}f)^{q(s)}\log P_{s,t}f.
  \end{align*}
  Under Hypothesis {\bf(H3)}, by Theorem \ref{cor-gradient}~(c), there
  exists a positive constant $c(s,t)$ such that
  $$\||\nabla^sP_{s,t}f|_s^2\|_{\infty}^{\mathstrut}\leq c(s,t)\|f\|^2_{\infty}.$$
  Moreover, $\|P_{s,t}f\|_{\infty}^{\mathstrut}\leq \|f\|_{\infty}^{\mathstrut}$ and
$$(P_{s,t}f)^{q(s)}\log^{+}(P_{s,t}f)\leq (P_{s,t}f)^{q(s)+1}\leq \|f\|_{\infty}^{q(s)+1}.$$
Combining all estimates above, we obtain
$\|(L_s+\partial_s)(P_{s,t}f)^{q(s)}\|_{\infty}^{\mathstrut}<\infty$. Now using
Lemma \ref{lem1}, we have
\begin{align*}
  \frac{\vd}{\vd s}&\mu_s((P_{s,t}f)^{q(s)})\\
                   &=\mu_s\l[L_s(P_{s,t}f)^{q(s)}-q(s)(P_{s,t}f)^{q(s)-1}(L_sP_{s,t}f)+q'(s)(P_{s,t}f)^{q(s)}\log P_{s,t} f\r]\\
                   &=q(s)(q(s)-1)\mu_s(|\nabla^sP_{s,t}f|_s^2(P_{s,t}f)^{q(s)-2})+q'(s)\mu_s\big((P_{s,t}f)^{q(s)}\log P_{s,t}f\big).
\end{align*}
Furthermore, for $\|P_{s,t}f\|_{q(s),s}$, we have
\begin{align}\label{add-eq-4}
  \frac{\vd}{\vd s}\|& P_{s,t}f\|_{q(s),s}\nonumber\\
                     &= \|P_{s,t}f\|_{q(s),s}^{-q(s)+1}(q(s)-1)\mu_s(|\nabla^sP_{s,t}f|_s^2(P_{s,t}f)^{q(s)-2})\nonumber\\
                     &\quad+\frac{q'(s)}{q(s)}\|P_{s,t}f\|_{q(s),s}^{-q(s)+1}\mu_s((P_{s,t}f)^{q(s)}\log P_{s,t}f)-\frac{q'(s)}{q(s)}\|P_{s,t}f\|_{q(s),s}\log \|P_{s,t}f\|_{q(s),s}.
\end{align}
Replacing $f$ in the log-Sobolev inequality \eqref{log-S-I} by
$f^{p/2}$, we get
$$\int f^p\log \l(\frac{f^p}{\|f^{p/{2}}\|_{2,s}^2}\r)\vd \mu_s\leq r \frac{p^2}{4}\int f^{p-2}|\nabla^sf|_s^2\vd \mu_s+\beta_s(r) \|f^{p/{2}}\|_{2,s}^2.$$
Now again replacing $f$ and $p$ by $P_{s,t}f$ and $q(s)$ in the
inequality above, respectively, we obtain
\begin{align*}
  \int (P_{s,t}f)^{q(s)}& \log (P_{s,t}f)\,\vd \mu_s-\|P_{s,t}f\|_{q(s),s}^{q(s)}\log \|P_{s,t}f\|_{q(s),s}
  \\
                        &\leq r \frac{q(s)}{4}\int (P_{s,t}f)^{q(s)-2}|\nabla^s P_{s,t}f|_s^2\,\vd \mu_s+\frac{\beta_s(r)}{q(s)}
                          \|P_{s,t}f\|_{q(s),s}^{q(s)}.
\end{align*}
Combining this with Eq.~\eqref{add-eq-4} yields
$$\frac{\vd}{\vd s}\|P_{s,t}f\|_{q(s),s}\leq \frac{\beta_s(r)q'(s)}{q(s)^2}\|P_{s,t}f\|_{q(s),s},\quad (s,t)\in \Lambda,$$
where $$q(s)={\e}^{4r^{-1}(t-s)}(p-1)+1,\quad q(t)=p.$$ It follows that
\begin{align}\label{Pf-eq-2}
  \|P_{s,t}f\|_{q(s), s}\leq & \exp{\l[\int_s^t\frac{\beta_u(r)q'(u)}{q(u)^2}\,\vd u\r]}\|f\|_{p,t}.
\end{align}
If $q(s)=q$, then $r=4(t-s)\l(\log (q-1)/(p-1)\r)^{-1}$. Taking this
$r$ into Eq.~\eqref{Pf-eq-2} yields
$$\|P_{s,t}f\|_{q,s}\leq \exp{\l[\int_s^t\frac{\beta_u(4(t-s)\l(\log (q-1)/(p-1)\r)^{-1})q'(u)}{q(u)^2}\,\vd u\r]}\|f\|_{p,t}.$$

Next, we prove ``$\text{(a)}\Rightarrow\text{(b)}$''. Suppose that
there exists $C_{p,q}(s,t)$ and $ 1<p<q$ such that
$$\|P_{s,t}\|_{(p,t)\rightarrow (q,s)}\leq C_{p,q}(s,t).$$
Recall the log-Sobolev inequality with respect to $P_{s,t}$,
\begin{equation}\label{eq-log-sobolev}
  P_{s,t}(f^2\log f^2)\leq 4\l[\int_s^t\e^{-2\int_r^tk(u)\,\vd u}\,\vd r\r]\, P_{s,t}|\nabla^tf|_t^2+P_{s,t}f^2\log (P_{s,t}f^2),\quad f\in C_0^{\infty}(M).
\end{equation}
From this and the fact that
$$\log^{+}(P_{s,t}f^2)\leq P_{s,t}f^2\leq
\|f\|_{\infty}^2,$$
we are able to integrate both sides of Eq.~\eqref{eq-log-sobolev} with
respect to $\mu_s$,
\begin{align}\label{mu-f-1}
  \mu_t(f^2\log f^2)\leq 4\int_s^t\e^{-2\int_r^tk(u)\,\vd u}\,\vd r\cdot \mu_t(|\nabla^tf|^2_t)+\mu_s(P_{s,t}f^2\log P_{s,t}f^2).
\end{align}
Now, we need to deal with the term $\mu_s(P_{s,t}f^2\log P_{s,t}f^2)$.
For any $h\in (0,1-\frac1{p})$, by the Riesz-Thorin interpolation
theorem, we get
\begin{align}\label{add-eq-2}
  \|P_{s,t}f\|_{q_h,s}\leq C_{p,q}(s,t)^{r_h}\|f\|_{p_h,t},\quad f\in L^p(M,\mu_s),
\end{align}
where $r_h=\frac{ph}{p-1}\in (0,1)$,
$\frac1{p_h}=1-r_h+\frac{r_h}{p}$ and
$\frac1{q_h}=1-r_h+\frac{r_h}q$, i.\,e.,
 $$r_h=\frac{ph}{p-1},\quad p_h=\frac1{1-h},\quad q_h=\l(1-\frac{p(q-1)}{q(p-1)}h\r)^{-1}.$$
 Set $\|f\|_{2,t}=1$. Then from Eq.~\eqref{add-eq-2}, we have
 $$\int (P_{s,t}|f|^{2(1-h)})^{q_h}\vd \mu_s\leq C_{p,q}(s,t)^{r_hq_h},$$
 which further implies
 \begin{align*}
   &\frac1{h}\l[\int (P_{s,t}|f|^{2(1-h)})^{q_h}\vd \mu_s-\l(\int P_{s,t}|f|^2\vd \mu_s\r)^{q_h/p_h}\r]\\
   &\quad=\frac1{h}\l(\int (P_{s,t}|f|^{2(1-h)})^{q_h}\vd \mu_s-1\r)\leq\frac1{h}\l(C_{p,q}(s,t)^{r_hq_h}-1\r).
 \end{align*}
 As
 $$\lim_{h\rightarrow
   0}\frac1{h}(C_{p,q}(s,t)^{r_hq_h}-1)=\frac{p}{p-1}\log
 C_{p,q}(s,t),$$ by dominated convergence, we obtain
 $$\frac{p(q-1)}{q(p-1)}\int P_{s,t}f^2\log P_{s,t}f^2\vd \mu_s-\int P_{s,t}(f^2\log f^2)\vd \mu_s\leq \frac{p}{p-1}\log C_{p,q}(s,t),$$
 or equivalently,
 $$\mu_s(P_{s,t}f^2\log P_{s,t}f^2)\leq \frac{q(p-1)}{p(q-1)}\mu_t(f^2\log f^2)+\frac{q}{q-1}\log C_{p,q}(s,t).$$
 Combining this with Eq.~\eqref{mu-f-1}, we arrive at
 \begin{align}\label{eq-9}
   \mu_t(f^2\log f^2)\leq \gamma_t(t-s) \mu_t(|\nabla^tf|_t^2)+\tilde{\beta}_t(t-s)
 \end{align}
 where $f\in C_0^{\infty}(M)$, $\|f\|_{2,t}=1$ and
 \begin{align}\label{eq-5}
   \gamma_t(t-s)=\frac{4p(q-1)}{(q-p)}\int_{t-(t-s)}^t\e^{-2\int_r^tk(u)\,\vd u}\,\vd r,\quad\tilde{\beta}_t(t-s)=\frac{pq}{(q-p)}\log C_{p,q}(s,t),
 \end{align}
 i.e. $\tilde{\beta}_t$ is a positive function on $(0,\infty)$ and
 $2\leq p\leq q$.  We complete the proof by letting $\gamma_t=r$
 and then
 \begin{equation*}\beta_t(r)=\tilde{\beta}_t(\gamma_t^{-1}(r)).\qedhere\end{equation*}
\end{proof}

Next, we study the ultraboundedness by using the super-log-Sobolev
inequality \eqref{log-S-I}.
\begin{theorem}
  Suppose that Hypothesis {\bf(H3)} holds. Let $(\mu_t)$ be an
  evolution system of measures associated with $P_{s,t}$.
  \begin{enumerate}[\rm(i)]
  \item  If the function $k$ in Hypothesis {\bf(H3)} is almost
    surely non-negative and $P_{s,t}$ satisfies
   $$\|P_{s,t}\|_{(2,t)\rightarrow \infty}\leq C_{2,\infty}(s,t),$$
   then Eq.~\eqref{log-S-I} holds for
   $\beta_s(r)=2\log C_{2,\infty}(s,s+\frac{r}{8})$.
 \item Conversely, assume Eq.~\eqref{log-S-I} holds for some
   positive non-increasing function
   $\beta: (0,+\infty)\rightarrow (0,+\infty)$, which is independent
   of $s$.  If there exists a function $r\in C([2,\infty))$ such that
     $$t_0:=\int_2^{\infty}\frac{r(p)}{p-1}\vd p<\infty,$$
     then for $t-s\geq t_0$, we have
     $$\|P_{s,t}\|_{(2,t)\rightarrow \infty}\leq \exp \l(\int_2^{\infty}\frac{\beta(r(p))}{p^2}\vd p\r).$$
   \end{enumerate}
 \end{theorem}
 \begin{proof}
  Letting $p=2$ and
   $q\rightarrow +\infty$ in   Eq.~\eqref{eq-9},  we know  from Eq.~\eqref{eq-5} that for $f\in C_0^{\infty}(M)$
   with $\|f\|_{2,t}=1$,
   \begin{align*}
     \mu_t(f^2\log f^2)
&\leq 8\l(\int_{s}^t\e^{-2\int_r^tk(u)\,\vd u}\,\vd r \r) \mu_t(|\nabla^tf|_t^2)+2\log C_{2,\infty}(s,t)\\
&\leq 8(t-s)\mu_t(|\nabla^tf|_t^2)+2\log C_{2,\infty}(s,s+(t-s)).
   \end{align*}
Let $r=8(t-s)$. We then prove (i) directly.

   Given $(s,t)\in \Lambda$.  Let $q$ and $N$ be two functions in
   $C^1((-\infty,t])$, which will be given later.  It follows from
   Eq.~\eqref{add-eq-4} that for $f\in C_c^{\infty}(M)$ such that $f>0$,
   \begin{align}\label{Pf-est-1}
     \frac{\vd }{\vd s}{\rm e}^{-N(s)}&\|P_{s,t}f\|_{q(s),s}\nonumber\\
     =&\frac{q'(s)}{q(s)}{\e}^{-N(s)}\|P_{s,t}f\|_{q(s),s}^{-q(s)+1}\bigg[\mu_s\l((P_{s,t}f)^{q(s)}\log P_{s,t}f\r)-\|P_{s,t}f\|^{q(s)}_{q(s),s}\log \|P_{s,t}f\|_{q(s),s}\nonumber\\
                                      &\quad +\frac{q(s)(q(s)-1)}{q'(s)}\mu_s\l(|\nabla^sP_{s,t}f|_s^2(P_{s,t}f)^{q(s)-2}\r)
                                        -N'(s)\frac{q(s)}{q'(s)}\|P_{s,t}f\|_{q(s),s}^{q(s)}\bigg].
\end{align}
  From this, and applying the super-log-Sobolev inequality \eqref{log-S-I} to $(P_{s,t}f)^{q(s)/2}$, we obtain
\begin{align*}
\frac{\vd }{\vd s}&{\rm e}^{-N(s)}\|P_{s,t}f\|_{q(s),s}\\
&\geq  {\e}^{-N(s)}\|P_{s,t}f\|_{q(s),s}^{-q(s)+1}\frac{q'(s)}{q(s)}\bigg[\l(\frac{q(s)(q(s)-1)}{q'(s)}+r\frac{q(s)}{4}\r)\mu_s\l(|\nabla^sP_{s,t}f|_s^2(P_{s,t}f)^{q(s)-2}\r)\\
&\quad+\l(\frac1{q(s)}\beta(r)-N'(s)\frac{q(s)}{q'(s)}\r)\|P_{s,t}f\|^{q(s)}_{q(s),s}\bigg].
\end{align*}
Let $q(s)$ and $N(s)$ be the solutions to the following equations respectively:
$$q'(s)=\frac{-4(q(s)-1)}{r\circ q(s)},\quad q(t)=2;$$
$$N'(s)=\frac{q'(s)\beta(r\circ q(s))}{q(s)^2},\quad N(t)=0.$$
It follows that:
\begin{align}\label{Pf-eq-1}
{\e}^{-N(s)}\|P_{s,t}f\|_{q(s),s}\leq \|f\|_{2,t}.
\end{align}
Define $t_0:=\int_2^{\infty}\frac{r(p)}{4(p-1)}\vd p<\infty.$ We claim
that $q(s)\rightarrow +\infty$ as $s\rightarrow t-t_0$, and then
$N(s)\rightarrow \int_2^{\infty}\frac{\beta(r(p))}{p^2}\vd p.$ Indeed,
$q(t-t_0)=\infty$ follows from the fact that
\begin{align*}
  \int_2^{q(t-t_0)}\frac{r(s)}{4(s-1)}\,\vd s=\int_t^{(t-t_0)} \frac{r\circ q(s)q'(s)}{4(q(s)-1)}\,\vd s=t_0:=\int_2^{\infty}\frac{r(s)}{4(s-1)}\,\vd s,
\end{align*}
i.e. $q(t-t_0)=\infty.$ By this and Eq.~\eqref{Pf-eq-1}, we have
\begin{equation*}
\|P_{t-t_0,t}\|_{(2,t)\rightarrow\infty}\leq \exp{\l(\int_2^{\infty}\frac{\beta(r(p))}{p^2}\vd p\r)}.\qedhere
\end{equation*}
\end{proof}

\subsection{Dimension-free Harnack inequality and boundedness of semigroup}
Next, we will use integrability of the Gaussian
function ${\e}^{\lambda \rho_t^2}$ (for $\lambda>0$ and $t\in I$) with respect to the families of measures
$(\mu_s)_{s\in I}$ or $(P_{s,t})_{(s,t)\in \Lambda}$ 
to give  another criterion which is equivalent to supercontractivity or
ultraboundedness.  To this end, we need the following preliminary result
which is a dimension-free Harnack-type estimate for $P_{s,t}$ (see \cite{Cheng15}).
\begin{lemma}\label{lem2}
  Assume that $\mathcal{R}_t^Z\geq k(t)$ holds for $t\in I$. For every
  $f\in C_b(M)$, $p>1$, $(s,t)\in \Lambda$ and $x,y\in M$, we have the inequality:
  \begin{equation}\label{Har-Ineq}
    |P_{s,t}f|^p(x)\leq \l(P_{s,t}|f|^p\r)(y)\exp\l(\frac{p}{4(p-1)}\l(\int_s^t\exp\l(2\int_s^rk(u)\,\vd u\r)\,\vd r\r)^{-1}\rho_s^2(x,y)\r).
  \end{equation}
\end{lemma}

The main result of this section is the following:
\begin{theorem}\label{th-2}
  Suppose that Hypothesis {\bf(H3)} holds. Let $(\mu_s)$ be the
  evolution system of measures. Then
  \begin{enumerate}[\rm(i)]
  \item $P_{s,t}$ is supercontractive with respect to $(\mu_s)$ if
    and only if $\mu_t\l(\exp(\lambda\rho_t^2)\r)<\infty$ for any
    $\lambda>0$ and $t\in I$;
  \item $P_{s,t}$ is ultrabounded with respect to $(\mu_s)$ if and
    only if $\l\|P_{s,t}\exp(\lambda\rho_t^2)\r\|^{\mathstrut}_\infty<\infty$ for
    any $\lambda>0$ and $(s,t)\in \Lambda$.
  \end{enumerate}
\end{theorem}

\begin{proof}
  $(a)$\ From Hypothesis {\bf(H3)} we know that
  $\mathcal{R}_t^Z\geq k(t)$ for $t\in I$. It follows that by the
  Harnack inequality \eqref{Har-Ineq}, we have that for
  $(s,t)\in \Lambda$, $p>1$ and $f\in C_b(M)$,
$$|P_{s,t}f|^p(x)\leq 
P_{s,t}|f|^p(y)\exp\l(\frac{p}{4(p-1)}\l(\int_s^t\exp\l(2\int_s^rk(u)\,\vd u\r)\vd r\r)^{-1}\rho_s(x,y)^2\r).$$
If $\mu_t(|f|^p)=1$, then
\begin{align}\label{est-P}
  1&\geq |P_{s,t}f(x)|^p\int \exp\l(-\frac{p}{4(p-1)}\l(\int_s^t\exp\l(2\int_s^rk(u)\,\vd u\r)\vd r\r)^{-1}\rho_s(x,y)^2\r)\mu_s(\vd y)\nonumber\\
  &\geq  |P_{s,t}f(x)|^p\mu_s(B_s(o,R))\exp {\l(-\frac{p(\rho_s(x)+R)^2}{4(p-1)}\l(\int_s^t\exp\l(2\int_s^rk(u)\,\vd u\r)\vd r\r)^{-1}\r)},
\end{align}
where $B_s(o,R):=\{y\in M:\rho_s(y)\leq R\}$.  Since $(\mu_s)$ is
compact,  there exists $R>0$, which may depend on $s$, such that
\begin{align*}
  \mu_s(B_s(o,R))&=\mu_s(\{x: \rho_s(x)\leq R\})\geq 1-\frac{\mu_s(\rho_s^2)}{R^2}\geq 1-\frac{H_2(s)}{R^2}\geq 2^{-p}.
\end{align*}
By this and Eq.~\eqref{est-P}, we arrive at
$$
1\geq |P_{s,t}f(x)|^p2^{-p}\exp{\l(-\l(\int_s^t\exp\l(2\int_s^rk(u)\vd
    u\r)\vd r\r)^{-1}\frac{p(\rho_s^2(x)+R^2)}{4(p-1)}\r)}
$$
which further implies
\begin{align}\label{est-P-1}|P_{s,t}f(x)|\leq 2\exp
  {\l(\l(\int_s^t\exp\l(2\int_s^rk(u)\,\vd u\r)\vd
    r\r)^{-1}\frac{\rho_s^2(x)+R^2}{4(p-1)}\r)},\quad s<t.
\end{align}
Therefore, we have
$$\|P_{s,t}f\|_{q,s}\leq \l\{\mu_s\l(\exp\left(q(c_1+c_2\rho_s^2)\right)\r)\r\}^{1/q}$$
for some positive constants $c_1, c_2$ depending on $s,t$.  Hence, if
$\mu_s(\exp(\lambda \rho_s^2))<\infty$ for any $\lambda>0$ and
$s\in I$, then $P_{s,t}$ is supercontractive,
i.e. $$\|P_{s,t}\|_{(p,t)\rightarrow (q,s)}<\infty$$ for any
$1<p<q<\infty$.

Conversely, if the semigroup $P_{s,t}$ is supercontractive, then
   by Theorem \ref{log-S-superbound}, we know that the family of super-log-Sobolev
   inequalities \eqref{log-S-I} holds. Now our first step is to prove
   $\mu_s(\e^{\lambda \rho_s})<\infty$ for any $s\in I$ and $\lambda>0$.
   Let $\rho_s^n=\rho_s\wedge n$ and $h_{s,n}(\lambda)=\mu_s(\exp{(\lambda \rho^n_s)})$. Taking
$\exp(\frac{\lambda}2\rho_s^n)$ into the super-log-Sobolev inequality
\eqref{log-S-I} above, we have
\begin{align*}
  \lambda h'_{s,n}(\lambda)&-h_{s,n}(\lambda)\log h_{s,n}(\lambda)\leq h_{s,n}(\lambda)\lambda^2\l(\frac{r}{4}+\frac{\beta_s(r)}{\lambda^2}\r).
\end{align*}
This implies
\begin{align}\label{add-eq-5}
  \l(\frac1{\lambda}\log h_{s,n}(\lambda)\r)'=\frac{\lambda h'_{s,n}(\lambda)-h_{s,n}(\lambda)\log h_{s,n}(\lambda)}{\lambda^2 h_{s,n}(\lambda)}\leq \frac{r}{4}+\frac{\beta_s(r)}{\lambda^2}.
\end{align}
Integrating both sides of Eq.~\eqref{add-eq-5} from $\lambda$ to
$2\lambda$, we obtain
\begin{align}\label{eq-1}
  h_{s,n}(2\lambda)\leq h_{s,n}^2(\lambda)\exp\l(\frac{r}2\lambda^2+\beta_s(r)\r).
\end{align}
By this and the fact that there exists a constant $M_s$ such that
$$\mu_s(\{\lambda \rho_s\geq M_s\})\leq \frac1{4}\exp\l(-\l(\frac{r}2\lambda^2+\beta_s(r)\r)\r),$$
it follows that:
\begin{align*}
  h_{s,n}(\lambda)&=\int_{\{\lambda \rho_s\geq M_s\}}{\e}^{\lambda \rho_s^n}\vd \mu_s+\int_{\{\lambda \rho_s<M_s\}}{\e}^{\lambda \rho_s^n} \vd \mu_s\\
              &\leq \mu_s(\{\lambda \rho_s\geq M_s\})^{1/2}\, \mu_s({\e}^{2\lambda \rho_s^n})^{1/2}+{\e}^{M_s}\\
              &\leq \l(\frac1{4}\exp\l(-\l(\frac{r}2\lambda^2+\beta_s(r)\r)\r)\r)^{1/2}\exp\l(\frac{r}{4}\lambda^2+\frac12\beta_s(r)\r)h_{s,n}(\lambda)+{\e}^{M_s}\\
              &\leq \frac12 h_{s,n}(\lambda)+{\e}^{M_s},
\end{align*}
which implies $h_{s,n}(\lambda)\leq 2{\e}^{M_s}$ for
$s\in I$. As $M_s$ is independent of $n$, letting $n$ go to infinity,  we obtain
$$\mu_s({\e}^{\lambda \rho_s})<\infty,\quad \text{for}\ \, s\in I.$$

Our second step is to prove $\mu_s(\e^{\lambda \rho_s^2})<\infty$ for all $s\in I$ and $\lambda>0$. Let $h_s(\lambda):=\lim_{n\rightarrow \infty} h_{s,n}(\lambda)$. Integrating both sides of Eq.~\eqref{add-eq-5} from $1$
to $\lambda$ and letting $n\rightarrow \infty$, we obtain
\begin{align}\label{eq-2}
  h_s(\lambda)\leq \exp\l(\lambda c_0(s)+\frac{r}{4}(\lambda^2-\lambda)+\beta_s(r)(1-\lambda)\r)
\end{align}
where $c_0(s):=\log \mu_s(\exp(\rho_s))$.  Now, we observe that for
any positive constant $\epsilon$,
$$\int_1^{\infty}h_s(\lambda)\e^{-(\frac{r}{4}+\epsilon)\lambda^2}\,\vd \lambda=\int_M\vd \mu_s\int_1^{\infty}\e^{\lambda \rho_s}\e^{-(\frac{r}{4}+\epsilon)\lambda^2}\,\vd \lambda<\infty.$$
On the other hand, it is easy to see that for $\epsilon>0$,
\begin{align*}
  \int_M&\vd \mu_s\int_1^{\infty}\e^{\lambda \rho_s}\e^{-(\frac{r}{4}+\epsilon)\lambda^2}\,\vd \lambda\\
  &=\int_M\exp\l(\rho_s^2/(r+4\epsilon)\r)\,\vd\mu_s\int_1^{\infty}\exp\l(-\l(\frac12\sqrt{r+4\epsilon}\lambda -\rho_s/\sqrt{r+4\epsilon}\r)^2\r)\vd \lambda\\
  &\geq \frac2{\sqrt{r+4\epsilon}}\int_M\exp\l(\rho_s^2/(r+4\epsilon)\r)\,\vd\mu_s \int_{\sqrt{r+4\epsilon}/2}^{\infty}\exp(-t^2)\,\vd t.
\end{align*}
By the arbitrariness of $r$, we obtain that there exists a number
$N_s$ such that for any $\lambda>0$,
$$\int {\e}^{\lambda \rho_s^2}\vd \mu_s<N_s,\quad s\in I,$$
which completes the proof of (i).

\noindent
(b) \ If $\|P_{s,t}\exp{[\lambda \rho_t^2]}\|_{\infty}^{\mathstrut}<\infty$ for any
$\lambda>0$ and $(s,t)\in \Lambda$, then  we know from Eq.~\eqref{est-P-1} that for any $p>1$ and $f\in C_b(M)$ satisfying $f>0$ and $\|f\|_{p,t}=1$,  
\begin{align*}
|P_{(s+t)/2,t}f(x)|\leq 2\exp{\l(\l(\int_{(s+t)/2}^t\exp\l(2\int_{(s+t)/2}^rk(u)\,\vd u\r)\,\vd r\r)^{-1}
\frac{\rho_{(s+t)/2}^2(x)+R^2}{4(p-1)}\r)}
\end{align*}
which implies that there exist constants $c_1$ and $c_2$ such that
$$\|P_{s,t}f\|_{\infty}\leq 2\l\|P_{s,(t+s)/2}\exp\l(2\l(c_1+c_2\rho_{(s+t)/2}^2(x)\r)\l(\int_{{(s+t)}/2}^t\!\exp\l(2\int_{{(s+t)}/2}^rk(u)\vd u\r)\vd r\r)^{-1}\r)\r\|_{\infty}^{\mathstrut}<\infty. $$
On the other hand, if $\|P_{s,t}\|_{(p,t)\rightarrow\infty}<\infty$ for all $p>1$,
then
$$\l\|P_{s,t}\exp(\lambda \rho_t^2)\r\|_{\infty}^{\mathstrut}
\leq \|P_{s,t}\|_{(p,t)\rightarrow \infty}\l\|\exp(\lambda \rho_t^2)\r\|_{p,t}<\infty$$
provided $\mu_t(\exp(\lambda \rho_t^2))$ is bounded for all $t\in I$.
Hence, it suffices to prove that
$\mu_t(\exp(\lambda \rho_t^2))<\infty$. Since $P_{s,t}$ is
ultrabounded, $P_{s,t}$ is supercontractive. Using Theorem \ref{th-2} (i), we
get $\mu_t(\exp(\lambda \rho_t^2))<\infty$. This completes the proof.
\end{proof}

\subsection{Other criteria on supercontractivity and ultraboundedness}

It is straightforward to check that Hypothesis {\bf(H3)} implies Hypothesis  {\bf(H2)} for
$\varphi(r)=r^2$, $r>0$.  As far as supercontractivity
and ultraboundedness of $P_{s,t}$ is concerned, we have the following results in terms 
of other types of space-time Lyapunov conditions.

\begin{theorem}\label{c:SH}
  Let $\gamma \in C((0,\infty))$ be a positive
  increasing function such that
  $\lim\limits_{r\rightarrow +\infty}\frac{\gamma(r)}{r}=+\infty$.
  \begin{enumerate}[\rm(i)]
  \item If
    \begin{align}\label{Ly-C}
      (L_t+\partial_t)\rho_t^2(x)\leq c-\gamma(\rho^2_t(x))
    \end{align}
    holds for $t\in I$, $c>0$ and $x\notin \Cut_t(o)$, then $P_{s,t}$
    has an evolution system of measures $(\mu_s)$ and $P_{s,t}$ is
    supercontractive with respect to $(\mu_s)$.
  \item If \eqref{Ly-C} holds for $\gamma$ such that
    $g_{\lambda}(r)=r\gamma(\lambda \log r)$ is convex on $(0,\infty)$
    and such that for any $\lambda>0$,
 $$\int_0^{\infty}\frac{\,\vd r}{r\gamma(\lambda \log r)}<\infty,$$
 then $P_{s,t}$ has an evolution system of
 measures $(\mu_s)$ and $P_{s,t}$ is ultrabounded with respect to
 $(\mu_s)$.
\item If  \eqref{Ly-C} holds for $\gamma(r)=\alpha r^{\delta}$, where
  $\alpha>0$ and $\delta>1$, then $P_{s,t}$ is ultrabounded with
  respect to $(\mu_s)$ and
$$\|P_{s,t}\|_{(2,t)\rightarrow \infty}\leq \exp\l(c(t-s)^{-\delta/(\delta-1)}\r)$$
holds for some constant $c>0$ and all $(s,t)\in \Lambda$.
\end{enumerate}
\end{theorem}

\begin{proof}
  It is an immediate consequence of Theorem \ref{nonep} that under condition \eqref{Ly-C}
  the process $X_{\displaystyle\bf .}$ is non-explosive up
  to time $T$.  The idea of following proof is similar to
  \cite[Corollary 5.7.6]{Wbook1}.  We include a proof for 
  convenience.
  
  (a)\ \ Let $X_t$ be a diffusion processes generated by $L_t$. Then
  by the It\^{o} formula,
  \begin{align}\label{Ito-e1}
    \vd \exp (\lambda \rho_t^2(X_t))
    =2\lambda\rho_t(X_t)\exp(\lambda \rho_t^2(X_t))\vd b_t+{\bf 1}_{\{X_t\notin \Cut_t(o)\}}(L_t+\partial_t)
    \exp(\lambda \rho_t^2(X_t))\,\vd t-\vd \ell_t,
  \end{align}
  where $b_t$ is a one-dimensional Brownian motion and $\ell_t$ an
  increasing process supported on $\{t\geq s: X_t\in \Cut_t(o)\}$. By
  Eq.~\eqref{Ly-C}, it follows that
  \begin{align}\label{e:est-rho-1-0}
    (L_t+\partial_t)\exp(\lambda\rho_t^2)
    &=\lambda \exp (\lambda \rho_t^2)(L_t+\partial_t)\rho_t^2+4\lambda^2\rho_t^2\exp(\lambda \rho_t^2)\notag\\
    &\leq \exp (\lambda \rho_t^2)(c-\gamma(\rho_t^2))+4\lambda^2\rho_t^2\exp(\lambda \rho_t^2)
  \end{align}
  holds outside $\Cut_t(o)$.  If
  $\limsup\limits_{r\rightarrow\infty}\frac{\gamma(r)}{r}=+\infty,$
  then there exist $c_1, c_2>0$ such that for each $t\in I$,
  $$(L_t+\partial_t)\l(\exp(\lambda \rho^2_t)-1\r)\leq
  c_1-c_2\l(\exp(\lambda \rho_t^2)-1\r)$$
  holds outside $\Cut_t(o)$.  According to Theorem
  \ref{invariant-measure-th}, there exists an evolution system of
  measures $(\mu_s)$ such that
  $\sup_{s\in I}\mu_s(\e^{\lambda \rho_s^2})<\infty$. We then obtain
  the first conclusion by Theorem \ref{th-2}~(i).

  (b)\ \ We use Theorem \ref{th-2}~(ii) to give the proof. So it
  suffices for us to check
  $\|P_{s,t}\exp(\lambda \rho_t^2)\|_{\infty}<\infty$.  By
  Eq.~\eqref{e:est-rho-1-0}, we have
  \begin{align}\label{e:est-rho-1}
    (L_t+\partial_t)\exp(\lambda\rho_t^2)
    &\leq \lambda \exp(\lambda \rho_t^2)(c-\gamma(\rho_t^2))+4\lambda^2\rho_t^2\exp{(\lambda \rho_t^2)}\nonumber\\
    &\leq \lambda \exp(\lambda \rho_t^2)(c_1-\gamma(\rho^2_t)/2),
  \end{align}
  where $c_1=0\vee \sup \l(c-\frac12\gamma(r)+4\lambda r\r)<\infty$.
  According to Eq.~\eqref{e:est-rho-1}, there exists a positive
  constant $C(\lambda)$ such that
  \begin{align}\label{e:est-rho-2}
    (L_t+\partial_t)\exp(\lambda \rho_t^2)\leq C(\lambda),
  \end{align}
  where $\lambda>0$.  For fixed $x\in M$, let
  $\theta_s(t):=\E^{(s,x)}\exp(\lambda \rho_t^2(X_t))$, $t\geq s$. We
  need to show that $\theta_s(t)$ is uniformly bounded. Since the set
  $\{t\in [s,T): X_t\in \Cut_t(o)\}$ is of measure zero, it follows
  from Eq.~\eqref{Ito-e1} and Eq.~\eqref{e:est-rho-2} that
$$\E^{(s,x)}[\lambda \rho_t^2(X_{t\wedge \tau_n})]\leq \exp(\lambda \rho_s^2(x))+C(\lambda)\E^{(s,x)}[t\wedge \tau_n-s],$$
where $\tau_n:=\inf \{t\in [s,T): \rho_t(X_t)\geq n\}$. Since
$\tau_n\uparrow T$ as $n\rightarrow \infty$, we further conclude that
$$\theta_s(t)\leq \exp(\lambda \rho_s(x)^2)+C(\lambda)(t-s)$$
for any $\lambda>0$ and $(s,t)\in \Lambda$.  In particular, here
$\theta_s$ is continuous and
$$M_t:=2\sqrt2\lambda\int_s^t\rho_r(X_r)\exp\left(\lambda
  \rho_r^2(X_r)\right)\vd b_r$$
is a square integrable martingale.  By Fubini's theorem,
along with Eq.~\eqref{Ito-e1} and Eq.~\eqref{e:est-rho-1}, we have
\begin{align*}
  \frac{\theta_s(t+r)-\theta_s(t)}{r}&\leq \frac{\lambda c_1}{r}\int_t^{t+r}\theta_s(u)\,\vd u
                                       -\frac{\lambda}{2r}\int_{t}^{t+r}\E^{(s,x)}\l[\exp\l(\lambda \rho_u^2(X_u)\r)\gamma(\rho_u^2(X_u))\r]\,\vd u\\
                                     &\leq \frac{\lambda c_1}{r}\int_t^{t+r}\theta_s(u)\,\vd u-\frac{\lambda}{2r}\int_t^{t+r}\theta_s(u)\gamma(\lambda^{-1}\log \theta_s(u))\,\vd u
\end{align*}
where the second inequality comes from the fact that for $\lambda>0$, the function 
$r\mapsto r\gamma(\lambda \log r)$ is convex for $r\geq 1$.
Therefore,
$$\theta_s'(t)\leq \lambda c_1\theta_s(t)-\frac{\lambda}2\theta_s(t)\gamma(\lambda^{-1}\log \theta_s(t)),\quad t\in [s,T).$$
Then, by a similar discussion as in the fixed metric case (see the
proof of \cite[Corollary 5.7.6]{Wbook1}), we obtain
\begin{align}\label{esti-theta}
  \theta_s(t)\leq G^{-1}(\lambda (t-s)/4)\vee c_2<\infty,
\end{align}
for some positive constant $c_2$ where
$$G(r):=\int_r^{\infty}\frac{\vd s}{s\gamma(\lambda^{-1}\log s)},\quad r>1.$$
In particular, for $\gamma(r)=\alpha r^{\delta}$ for $\delta>1$ and
$\alpha>0$, we have
$$G(r)=\frac{\lambda^{\delta}}{\alpha(\delta-1)}(\log r)^{1-\delta}.$$
Combining this with Eq.~\eqref{esti-theta}, we complete the proof of
(ii) and (iii).
\end{proof}

\proof[Acknowledgements]This work has been supported by the Fonds
National de la Recherche Luxembourg (FNR) under the OPEN scheme
(project GEOMREV O14/7628746). The first named author acknowledges
support by NSFC (Grant No.~11501508) and Zhejiang Provincial Natural
Science Foundation of China (Grant No. LQ16A010009).

\bibliographystyle{amsplain}%
\bibliography{evolution5}

\end{document}